\documentclass[a4paper,10pt]{amsart}
\usepackage{amsfonts,amssymb,amsmath} 
\usepackage{xcolor}
\usepackage{graphicx}
\usepackage{bbm}

\newcommand{\R}{\mathbbm{R}}

\newcommand{\pl}{\partial}

\newcommand{\Z}{\mathbbm{Z}}

\newcommand{\s}{\sigma}
\newcommand{\la}{\lambda}
\newcommand{\mi}{t\wedge s}
\newcommand{\ind}{\mathbbm{1}}
\newcommand{\e}{\varepsilon}

\newcommand{\Sp}{\mathbb{S}^{1}}
\newcommand{\Spd}{\mathbb{S}^{d-1}}

\newcommand{\xis}{\xi_{\sigma}}
\newcommand{\ls}{\lambda_{\sigma}}

\newtheorem{main}{Theorem}

\newtheorem{Lemma}{Lemma}
\newtheorem{Prop}{Proposition}

\begin{document}

\title[Periodic homogenization for linear Boltzmann equation]{Homogenization\\ 
of the linear Boltzmann equation\\ in a domain with a periodic distribution of holes}

\author[E. Bernard]{Etienne Bernard}
\address[E. B.]%
{Ecole polytechnique\\
Centre de math\'ematiques L. Schwartz\\
F91128 Palaiseau cedex\\
\& Universit\'e P.-et-M. Curie\\
Laboratoire J.-L. Lions, BP 187\\
75252 Paris cedex 05} 

\email{etienne.bernard@math.polytechnique.fr}

\author[E. Caglioti]{Emanuele Caglioti}
\address[E. C.]%
{Universit\`a di Roma ``La Sapienza"\\
Dipartimento di Matematica\\
p.le Aldo Moro 5\\
00185 Roma, Italia} 

\email{caglioti@mat.uniroma1.it}

\author[F. Golse]{Fran\c cois Golse}
\address[F. G.]%
{Ecole polytechnique\\
Centre de math\'ematiques L. Schwartz\\
F91128 Palaiseau cedex\\
\& Universit\'e P.-et-M. Curie\\
Laboratoire J.-L. Lions, BP 187\\
F75252 Paris cedex 05} 

\email{francois.golse@math.polytechnique.fr}

\keywords{Linear Boltzmann equation, Periodic homogenization, Periodic Lorentz gas, 
Renewal equation}

\subjclass[2000]{82C70, 35B27 (82C40, 60K05)}

\begin{abstract}
Consider a linear Boltzmann equation posed on the Euclidian plane with a periodic
system of circular holes and for particles moving at speed $1$. Assuming that the 
holes are absorbing --- i.e. that particles falling in a hole remain trapped there forever, 
we discuss the homogenization limit of that equation in the case where the reciprocal
number of holes per unit surface and the length of the circumference of each hole 
are asymptotically equivalent small quantities. We show that the mass loss rate due 
to particles falling into the holes is governed by a renewal equation that involves the 
distribution of free-path lengths for the periodic Lorentz gas. In particular, it is proved
that the total mass of the particle system at time $t$ decays exponentially fast as
$t\to+\infty$. This is at variance with the collisionless case discussed in [Caglioti, E.,
Golse, F., Commun. Math. Phys. \textbf{236} (2003), 199--221], where the total mass
decays as $\hbox{Const.}/t$ as $t\to+\infty$.
\end{abstract}

\maketitle


\section{Introduction}

The homogenization of a transport process describing the motion of particles in a
system of fixed obstacles --- such as scatterers, or holes --- leads to very different
results according to whether the distribution of obstacles is periodic or random.
Before describing the specific problem analyzed in the present work, we recall a
few results recently obtained on a more complicated, and yet related problem.

An important example of the phenomenon mentioned above is the Boltzmann-Grad
limit of the Lorentz gas. The Lorentz gas is the dynamical system corresponding to 
the free motion of a single point particle in a system of fixed spherical obstacles,
assuming that each collision of the particle with any one of the obstacles is purely
elastic. Since the particle is not subject to any external force, we assume without
loss of generality that its speed is $1$. The Boltzmann-Grad limit is the scaling limit 
where the obstacle radius and the reciprocal number of obstacles per 
unit volume vanish in such a way that the average free path length of the particle 
between two consecutive collisions with the obstacles is of the order of unity.

Call $f(t,x,v)$ the particle distribution function in phase space in that scaling limit 
--- in other words,
the probability that the particle be located in an infinitesimal volume $dx$ around
the position $x$ with direction in an infinitesimal element of solid angle $dv$
around the direction $v$ at time $t\ge 0$ is $f(t,x,v)dxdv$.

In the case of a random system of obstacles --- more precisely, assuming that the
obstacles centers are independent and distributed in the $3$-dimensional Euclidian 
space under Poisson's law --- Gallavotti proved in \cite{Gallavotti1969,Gallavotti1972} 
(see also \cite{Gallavotti1999} on pp. 48--55) that the average of $f$ over obstacle 
configurations (i.e. the mathematical expectation of $f$) is a solution of the linear 
Boltzmann equation
$$
(\pl_{t}+v\cdot\nabla_{x}+\s)f(t,x,v)=
\frac{\s}{\pi}\int_{\omega\cdot v>0\atop|\omega|=1}
	f(t,x,v-2(\omega\cdot v)\omega)\omega\cdot vd\omega\,.
$$

If, on the contrary, the obstacles are periodically distributed --- specifically, if they
are centered at the vertices of a cubic lattice --- the limiting particle distribution 
function $f$ cannot be the solution of any linear Boltzmann equation of the form
$$
(\pl_{t}+v\cdot\nabla_{x}+\s)f(t,x,v)=\s\int_{|w|=1}p(v|w)f(t,x,w)dw\,,
$$
where $p$ is a continuous, symmetric transition probability density on the unit
sphere: see \cite{Golse2008} for a complete proof of this negative result, based on 
earlier estimates on the distribution of free path lengths for the periodic Lorentz
gas \cite{BGW,GW}. 

The correct limiting equation for the Boltzmann-Grad limit of the periodic Lorentz
gas was found only very recently: see \cite{CagliotiGolse08,MarklofStromb2008}.
In the $2$-dimensional case, the most striking feature of the theory presented in 
these references, is that the limiting equation is set on an \textit{extended phase space}
involving not only the particle position $x$ and direction $v$, as in all classical 
kinetic models, but also the (rescaled) distance $s$ to the next collision point with 
the obstacles and the impact parameter $h$ at this next collision point. 

The particle motion is described in terms of its distribution function in this extended 
phase space, $F\equiv F(t,x,v,s,h)$, which is governed by an equation of the form
$$
(\pl_{t}+v\cdot\nabla_{x}-\pl_{s})F(t,x,v,s,h)
	=\int_{-1}^1P(s,h|h')F(t,x,R[\pi-2\arcsin(h')]v,0,h')dh'
$$
where $R[\theta]$ designates the rotation of an angle $\theta$, and $P(s,h|h')$ is a 
nonnegative integral kernel whose explicit expression is given in \cite{CagliotiGolse08}
but is of little interest for the present discussion. The particle distribution function in the
classical phase space of kinetic theory is recovered in terms of $F$ by the following 
formula:
$$
f(t,x,v)=\int_0^{+\infty}\int_{-1}^1F(t,x,v,s,h)dhds\,.
$$
However, the particle distribution function $f$ itself does not satisfy a linear Boltzmann 
equation in closed form.

Loosely speaking, in the case of a periodic distribution of obstacles, the particle
``feels" the correlations between the obstacles, since its trajectory consists of
segments of maximal length avoiding the obstacles. This explains the need for
an extended phase space in order to describe the Boltzmann-Grad limit of the
Lorentz gas, in the periodic case. In the random case studied by Gallavotti, the 
obstacles centers are assumed to be independent, which reduces the complexity
of the limiting dynamics.

\smallskip
In the present work, we shall study a much simpler homogenization problem, 
which can be formulated as follows:

\smallskip
\noindent
\textbf{Problem.} Consider a system of point particles whose distribution function 
is governed by a linear Boltzmann equation. The particles are assumed to move 
in a periodic system of holes. Describe the asymptotic behavior of the total mass
of the particle system in the long time limit, assuming that the radius of the holes
and their reciprocal number per unit volume vanish so that the average distance
between holes is of the order of $1$.

\smallskip
Although the underlying dynamics in this problem is a lot simpler than that of the 
Lorentz gas, the homogenized equation is also set on an extended phase space, 
analogous to the one described above. 

A we shall see, the mathematical derivation of the homogenized equation in the
extended phase space for the problem above involves only very elementary arguments 
from functional analysis --- at variance with the case of the Boltzmann-Grad limit 
of the Lorentz gas, which requires a fairly detailed knowledge of particle trajectories.

\section{The model}

We consider the monokinetic, linear Boltzmann equation
\begin{equation}
\label{Boltz}
\pl_{t}f_{\e}+v\cdot\nabla_{x}f_{\e}+\s(f_{\e}-Kf_{\e})=0
\end{equation}
in space dimension 2.

The unknown function $f(t,x,v)$ is the density at time $t\in\R_{+}$ of particles
with velocity $v\in\Sp$, located at $x\in \R^{2}$. For each $\phi\in L^2(\Sp)$, 
we denote
$$
K\phi(v):=\frac{1}{2\pi}\int_{\Sp}k(v,w)\phi(w)dw,
$$ 
where $dw$ is the uniform measure (arc length) on the unit circle $\Sp$. We 
henceforth assume that
\begin{equation}
\label{kdef}
\begin{aligned}
k\in L^{2}(\Sp\times\Sp)\,,\quad&k(v,w)=k(w,v)\geq0\ \hbox{ a.e. in\ }v,w\in\Sp
\\
&\mbox{and\ }\tfrac{1}{2\pi}\int_{\Sp}k(v,w)dw=1\ a.e.\ in\ v\in\Sp.
\end{aligned}
\end{equation}
The case of isotropic scattering, where $k$ is a constant, is a classical model in 
the context of Radiative Transfer. Likewise, the case of Thomson scattering in
Radiative Transfer involves the integral kernel
$$
k(v,w)=\tfrac3{16}(1+(v\cdot w)^2)
$$
--- see for instance chapter I, \S 16 of \cite{Chandra60}. Finally, the collision 
frequency is a constant $\s > 0$. 

The linear Boltzmann equation (\ref{Boltz}) is set on the spatial domain $Z_{\e}$, 
i.e. the space $\R^{2}$ with a periodic system of holes removed:
$$
Z_{\e}:=\left\lbrace x\in\R^{2}\,|\,\hbox{dist}(x,\e\mathbbm{Z}^{2})>\e^{2}\right\rbrace\,. 
$$ 
We assume an absorption boundary condition on $\pl Z_{\e}$:
$$
f_{\e}=0\ \mbox{for\ }(t,x,v)\in\R^{*}_{+}\times\partial Z_{\e}\times\Sp,
	\ \mbox{whenever\ }v\cdot n_{x}>0\,,
$$ 
where $n_{x}$ denotes the inward unit normal vector to $Z_{\e}$ at the point
$x\in\pl Z_{\e}$. This condition means that a particle falling into any one of the 
holes remain there forever.

The same problem could of course be considered in any space dimension. Notice
however that, in space dimension $N\ge 2$, the appropriate scaling, analogous 
to the one considered here, would be to consider holes of radius $\e^{N/(N-1)}$ 
centered at the points of the cubic lattice $\e\mathbbm{Z}^N$ --- see for instance
\cite{BGW,GW}. Most of the arguments considered in the present paper can be
adapted without change to the higher dimensional case, except that the expression
of one particular coefficient appearing in the homogenized equation is not yet known
explicitly at the time of this writing.

The most natural question related to the dynamics of the system above is the
asymptotic behavior of the total mass of the particle system in the small obstacle
radius $\e\ll 1$ and long time limit.

The last two authors  have considered in \cite{Golse-Caglioti} the non-collisional 
case ($\s=0$) and proved that, in the limit as $\e\to 0^+$, the solution $f_{\e}$ 
converges in $L^{\infty}(\R_{+}\times\R^{2}\times\Sp)$ weak-* to a solution $f$ of 
the following non-autonomous equation:
\begin{equation}
\label{Boltza}
\pl_{t}f+v\cdot\nabla_{x}f=\frac{\dot{p}(t)}{p(t)}f\,,
\end{equation}
where $p$ is a positive decreasing function defined below. In that case, the
total mass of the particle system decays like $\hbox{Const.}/t$ as $t\to+\infty$.

Observe that, starting from the free transport equation, we obtain a non-autono- mous 
(in time) equation in the small $\e$ limit. In particular, the solution of equation (\ref{Boltza}) 
cannot be given by a semigroup in a function space such as $L^{p}(\R^{2}_{x}\times\Sp_{v})$.
As we shall see, the homogenization of the linear Boltzmann equation in the collisional
case ($\s>0$) leads to an even more spectacular change of structure in the equivalent
equation obtained in the limit. 

The work of the last two authors \cite{Golse-Caglioti} relies upon an explicit computation of 
the solution of the free transport equation, where the effect of the system of holes is handled 
with continued fraction techniques. In the present paper, we investigate the analogous 
homogenization problem in the collisional case ($\s>0$). As we shall see, there is no explicit 
representation formula for the solution of the linear Boltzmann equation, other than the one 
based on the transport process, a particular stochastic process, defined for example in 
\cite{PapanicoBAMS75}. 

This representation formula was used in a previous work of the first author \cite{EBMajorMass}, 
who established a uniform in $\e$ upper bound for the total mass of the particle system by 
a quantity of the form $\hbox{Const.}e^{-a_\s t}$ for some $a_\s>0$. This exponential decay 
is quite remarkable: indeed, there is a ``phase transition" between the collisionless case in 
which the total mass decays algebraically as $t\to+\infty$, and the collisional case in 
which the total mass decays at least exponentially fast in that same limit. 

In the present paper, we further investigate this phenomenon and show that the exponential
decay estimate found in \cite{EBMajorMass} is sharp, by giving an asymptotic equivalent of 
the total mass of the particle system in the small $\e$ limit as $t\to+\infty$.

Instead of the semi-explicit representation formula by the transport process, our argument
is based on the very special structure of the homogenized problem. The key observation
in the present work is that this homogenized problem involves a renewal equation, for which 
exponential decay is a classical result that can be found in classical monographs such as 
\cite{Feller}.

\section{The main results}

First, we recall the definition of the free path length in the direction $v$ for a particle starting 
from $x$ in $Z_{\e}$:
\begin{equation}
\label{length}
\tau_{\e}(x,v):=\text{inf}\left\lbrace t>0\,|\,x-tv\in\pl Z_{\e}\right\rbrace\,.
\end{equation}

The distribution of free path length has been studied in \cite{BGW,GW,Golse-Caglioti,BZ}.
In particular, it is proved that, for each arc $I\subset\Sp$ and each $t\ge 0$, one has
\begin{equation}
\label{BZLim}
\hbox{meas}(\{(x,v)\in(Z_{\e}\cap[0,1]^2)\times I\,|\,\e\tau_{\e}(x,v)>t\})\to p(t)|I|
\end{equation}
as $\e\to 0^+$, where $|I|$ denotes the length of $I$ and the measure considered in the
statement above is the uniform measure on $[0,1]^2\times\Sp$.

The following estimate for $p$ can be found in \cite{BGW}: there exist $C,C'>0$ such 
that, for all $t\geq 1$:
\begin{equation}
\label{BGWUnifBound}
\frac{C}{t}\leq\hbox{meas}(\{(x,v)\in(Z_{\e}\cap[0,1]^2)\times I\,|\,\e\tau_{\e}(x,v)>t\})\leq \frac{C'}{t}
\end{equation}
uniformly as $\e\to 0^+$, so that
\begin{equation}
\label{BGWBound}
\frac{C}{t}\leq p(t)\leq \frac{C'}{t}\,.
\end{equation}
In \cite{BZ} F. Boca and A. Zaharescu have obtained an explicit formula for $p$: 
\begin{equation}
\label{BZtwo}
p(t)=\int_{t}^{+\infty}(\tau-t)\Upsilon(\tau)d\tau\,,
\end{equation}
where the function $\Upsilon$ is expressed as follows: 
\begin{equation}
\label{BZone}
\Upsilon(t) = \frac{24}{\pi^{2}}\left\{
\begin{array}{ll}
1 & \mbox{if }\ t\in(0,\tfrac12], 
\\	\\
\frac{1}{2t}+2(1-\frac{1}{2t})^{2}\ln(1-\frac{1}{2t})-\frac{1}{2}(1-\frac{1}{t})^{2}\ln|1-\frac{1}{t}|  
	& \mbox{if }\ t\in(\tfrac12,+\infty)\,.
\end{array}
\right.
\end{equation}

\smallskip
This is precisely at this point that the case of space dimension $2$ differs from the higher 
dimensional case. Indeed, in space dimension higher than $2$, the existence of the limit
(\ref{BZLim}) has been proved in \cite{MarklofStromb2007}, while the uniform estimate 
analogous to (\ref{BGWUnifBound}) is to be found in \cite{GW}. However, no explicit formula
analogous to (\ref{BZtwo}) is known in that case, at least at the time of this writing. We have 
chosen to treat in the present paper only the case of the square lattice in space dimension
$2$ as it is the only case where the limit (\ref{BZLim}-\ref{BZtwo}) is known completely.

\begin{figure}
\label{Fig-p}
\begin{center}

\includegraphics[width=6.8cm]{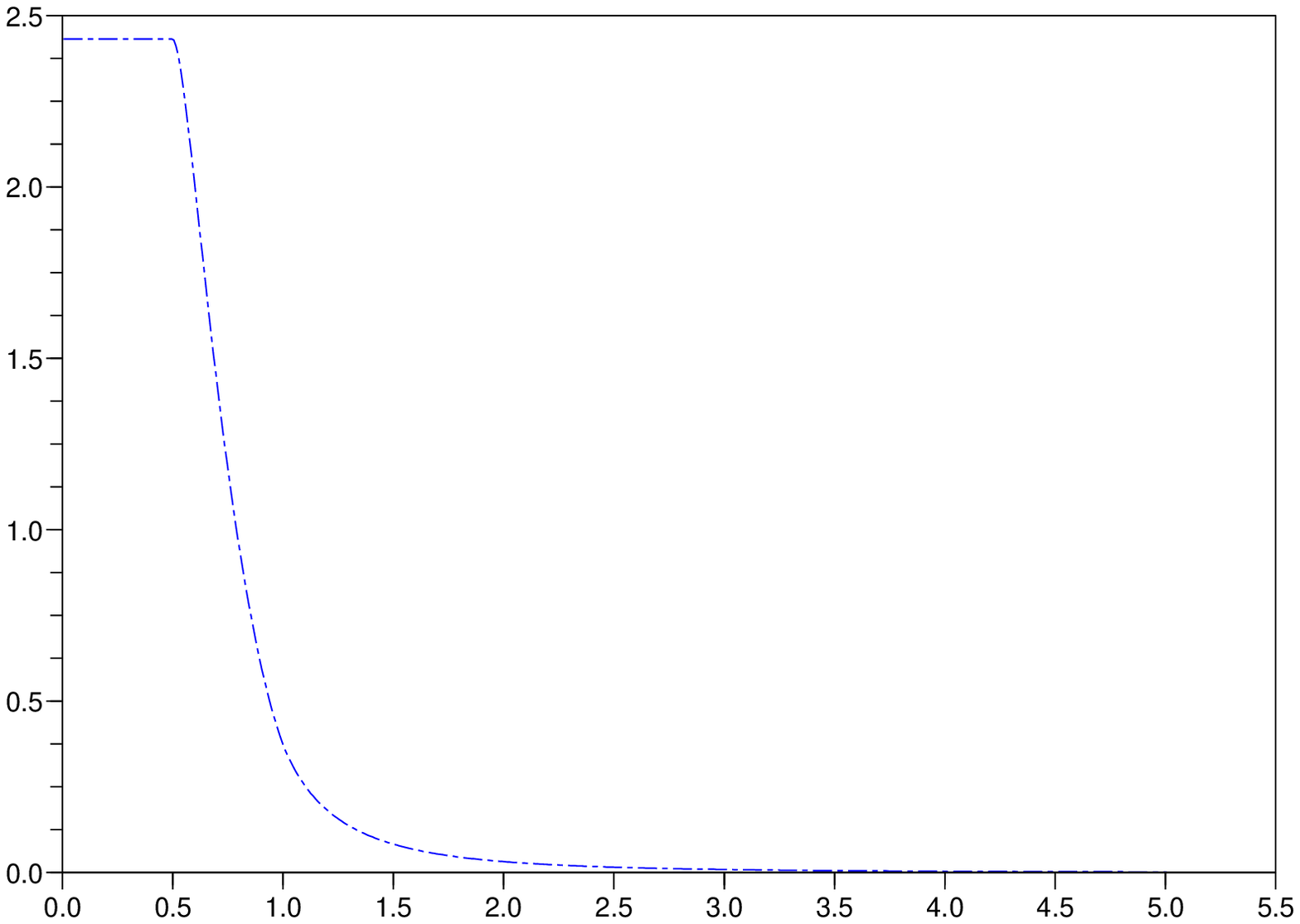}\!\!\!!\!\!\!\!\!\!\!\includegraphics[width=6.8cm]{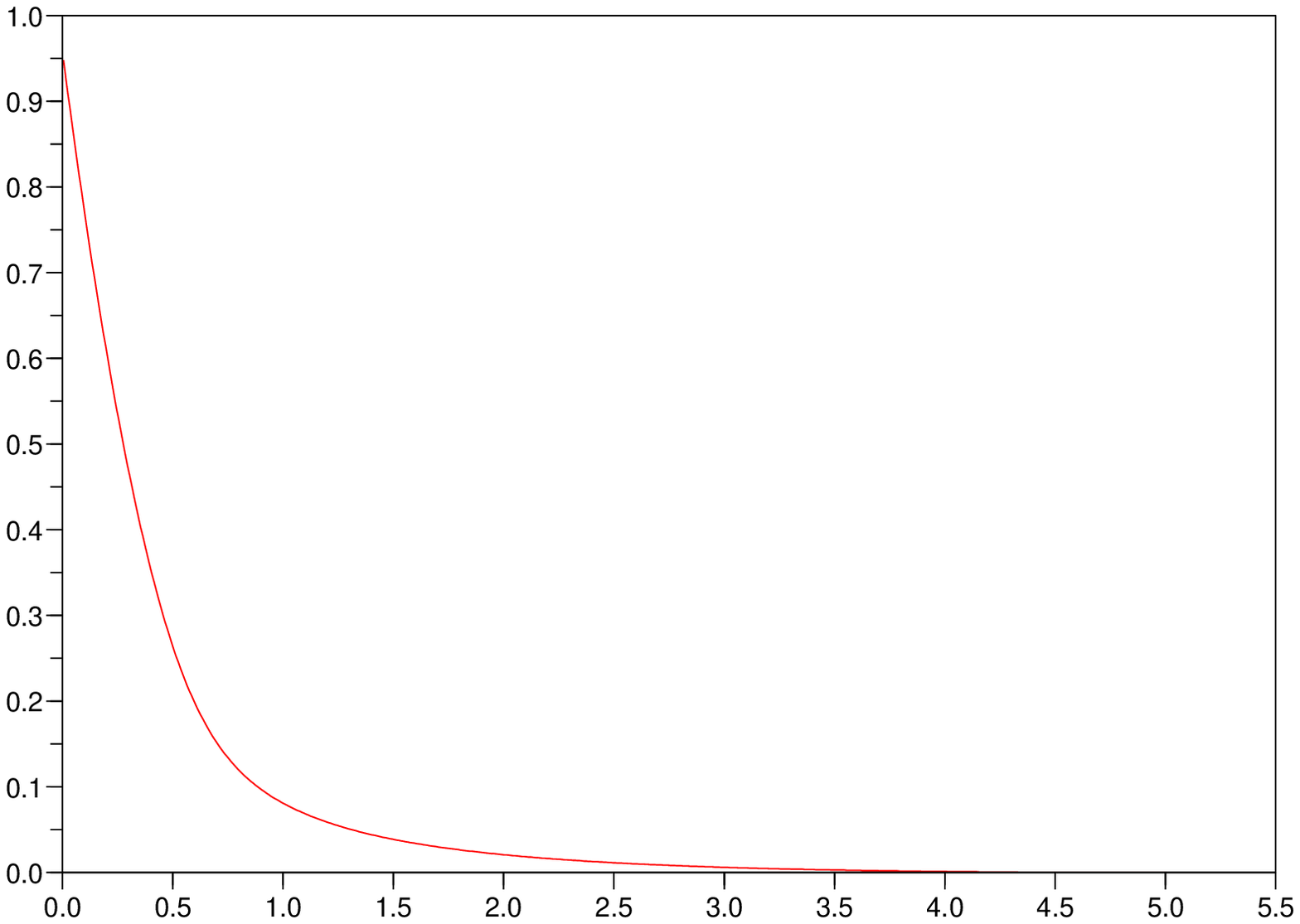}

\caption{The graphs of $\Upsilon$ (left) and of  $p$ (right)}
\end{center}
\end{figure}

Throughout this paper, we assume that the initial data of ($\Xi_\e$)  satisfies the 
assumption
\begin{equation}
\label{condinitiale}
f^{in}\geq0\ \mbox{on}\ \R^{2}\times\Sp\ \mbox{and} 
	\iint_{\R^{2}\times\Sp}f^{in}(x,v)dxdv+\sup_{(x,v)\in\R^{2}\times\Sp}f^{in}(x,v)<+\infty.
\end{equation}
For each $0< \e\ll 1$, let $f_{\e}$ be the (mild) solution of the initial boundary value 
problem
$$
(\Xi_{\e})\left\{
\begin{array}{ll}
\pl_{t}f_{\e}+v\cdot\nabla_{x}f_{\e}+\s(f_{\e}-Kf_{\e})=0, &\ (x,v)\in Z_{\e}\times\Sp,t>0,
\\	\\
f_{\e}=0,\ \mbox{if\ }v\cdot n_{x}>0,&\ (x,v) \in \pl Z_{\e}\times\Sp,
\\	\\
f_{\e}(0,x,v)=f^{in}(x,v),&\ (x,v) \in Z_{\e}\times\Sp.
\end{array}
\right.
$$
The classical theory of the linear Boltzmann equation guarantees the existence and
uniqueness of a mild solution $f_\e$ of the problem ($\Xi_\e$) satisfying 
\begin{equation}
\label{XiBounds}
\begin{aligned}
0\le f_\e(t,x,v)\le\sup_{(x,v)\in\R^2\times\Sp}f^{in}(x,v)\quad\hbox{ a.e. on }\R_+\times Z_\e\times\Sp\,,
\\
\iint_{Z_\e\times\Sp}f_\e(t,x,v)dxdv\le\iint_{\R^2\times\Sp}f^{in}(x,v)dxdv\,.
\end{aligned}
\end{equation}

Consider next $F:=F(t,s,x,v)$ the solution of the Cauchy problem
$$
(\Sigma)\left\{
\begin{array}{ll}
\pl_{t}F+v\cdot\nabla_{x}F+\pl_{s}F=-\s F+\frac{\dot{p}}{p}(\mi)F,\ 
	& t,s>0, (x,v)\in\R^{2}\times\Sp\,,
\\	\\
F(t,0,x,v)=\s\displaystyle\int_{0}^{+\infty}KF(t,s,x,v)ds,\ & t>0, (x,v)\in\R^{2}\times\Sp\,,
\\	\\
F(0,s,x,v)=\s e^{-\s s}f^{in}(x,v),\ & s>0, (x,v)\in\R^{2}\times\Sp\,,
\end{array}
\right.$$
with the notation $\mi:=\min(t,s)$. Notice that $F$ is a density defined on the extended 
phase space:
$$
\left\lbrace (s,x,v)|s\geq0,x\in\R^{2},v\in\Sp\right\rbrace
$$
involving the extra variable $s$, whose interpretation is given below.

\smallskip
Henceforth, we shall frequently need to extend functions defined a.e. on $Z_\e$ by $0$ 
inside the holes (that is, in the complement of $\overline{Z_\e}$). We therefore introduce 
the following piece of notation.

\paragraph{\underline{Definition}:} 
For each function $\varphi\equiv\varphi(x)$ defined a.e. on $Z_\e$, we denote
$$
\left\lbrace\varphi\right\rbrace(x) 
=\left\{
\begin{array}{ll} 
\varphi(x) &\ \mbox{if } x\in Z_{\e},
\\
0 &\ \mbox{if } x\notin\overline{Z_{\e}},
\end{array}
\right.
$$
We use the same notation $\{f_\e\}$ or $\{F_\e\}$ to designate the same extension by
$0$ inside the holes for functions defined on cartesian products involving $Z_\e$ as
one of their factors, such as $\R_+\times Z_\e\times\Sp$ in the case of $f_\e$, and
$\R_+\times\R_+\times Z_\e\times\Sp$ in the case of $F_\e$.

\smallskip
Our first main main result is

\begin{main}\label{premiertheoreme}
Under the assumptions above,
$$
\{f_{\e}\}\rightharpoonup\int_{0}^{+\infty}Fds
$$
in $L^{\infty}(\R_{+}\times\R^{2}\times\Sp)$ weak-$*$ as $\e\rightarrow0^{+}$, where
$F$ is the unique (mild) solution of ( $\Sigma$).
\end{main}

Notice that the limit of the (extended) distribution function of the particle system is indeed
defined in terms of the solution $F$ of the homogenized integro-differential equation 
($\Sigma$). However, it does not seem that the limit of $\{f_\e\}$ itself satisfies any natural
equation.

\smallskip
Next we discuss the asymptotic decay as $t\to+\infty$ of the total mass of the particle
system in the homogenization limit $\e\ll 1$. Obviously, the particle system loses mass 
due to particles falling into the holes.

In order to do so, we introduce the quantity:
$$
m(t,s):=\tfrac{1}{2\pi}\iint_{\R^{2}\times\Sp} F(t,s,x,v)dxdv\,.
$$
A key observation in our work is that $m$ is the solution of a renewal type PDE, as 
explained in the next proposition.

\begin{Prop}\label{renouv}
Denote
$$
B(t,s)=\s-\frac{\dot{p}}{p}(\mi)\,,
$$
and assume that $f^{in}$ satisfies the condition (\ref{condinitiale}). 

Then the renewal PDE
$$
\left\{
\begin{array}{ll}
\pl_{t}\mu(t,s)+\pl_{s}\mu(t,s)+B(t,s)\mu(t,s)=0,\ &\ t,s>0\,,
\\	\\	
\mu(t,0)=\s\displaystyle\int_{0}^{+\infty}\mu(t,s)ds,\ &\ t>0\,,
\\	\\	
\mu(0,s)=\s e^{-\s s},\ &\ s>0\,,
\end{array}
\right.
$$
has a unique mild solution $\mu\in L^\infty([0,T];L^1(\R_+))$ for all $T>0$.

Moreover, one has
$$
m(t,s)=\frac{\mu(t,s)}{2\pi}\iint_{\R^{2}\times\Sp}f^{in}(x,v)dxdv
$$
a.e. in $(t,s)\in\R_+\times\R_+$.
\end{Prop}

Renewal equations are frequently met in many different contexts. For instance they are 
used as a mathematical model in biology to study the dynamics of structured populations. 
The interested reader can consult \cite{Iannelli} or \cite{Thieme} for more information on
this subject. 

Consider next the quantity:
\begin{equation}\label{Mdef}
M(t):=\tfrac{1}{2\pi}\int_{0}^{+\infty}\iint_{\R^{2}\times\Sp}F(t,s,x,v)dxdvds
	=\int_{0}^{+\infty}m(t,s)ds\,.
\end{equation}
As explained in the theorem below, $M(t)$ is the total mass at time $t$ of the particle
system in the limit as $\e\to 0^+$; besides, the asymptotic behavior of $M$ as $t\to+\infty$
is a consequence of the renewal PDE satisfied by the function $(t,s)\mapsto m(t,s)$. 

\begin{main}
\label{secondtheorem}
Under the same assumptions as in theorem $\ref{premiertheoreme}$,
\begin{enumerate}
\item the total mass
$$
\tfrac{1}{2\pi}\iint_{Z_\e\times\Sp}f_{\e}(t,x,v)dxdv\rightarrow M(t)
$$
in $L^1_{loc}(\R_+)$ as $\e\to 0^+$, and a.e. in $t\ge 0$ after extracting a subsequence of
$\e\to 0^+$;
\item the limiting total mass is given by the representation formula
$$
M(t)=\tfrac{1}{2\pi\s}\iint_{\R^{2}\times\Sp}f^{in}(x,v)dxdv\sum_{n\geq1}\kappa^{*n}(t),\ t>0
$$
with 
$$
\kappa(t):=\s e^{-\s t}p(t)\ind_{t\geq0}\,,
	\quad\kappa^{*n}:=\underbrace{\kappa*\cdots*\kappa}_{\hbox{$n$ factors}}
$$ 
and $*$ denoting as usual the convolution product on the real line;
\item for each $\s>0$, there exists $\xis\in (-\s,0)$ such that
$$ 
M(t)\sim C_{\s}e^{\xis t}\ \mbox{as}\ t\rightarrow+\infty
$$ 
with 
$$
C_{\s}:=\tfrac{1}{2\pi\s}\frac{\displaystyle\iint_{\R^{2}\times\Sp} f^{in}(x,v)dxdv}
	{\displaystyle\int_{0}^{\infty} tp(t)e^{-(\s+\xis)t}dt}\,;
$$ 
\item finally, the exponential mass loss rate $\xis$ satisfies
$$
\xis\sim-\s\ \mbox{as}\ \s\rightarrow0^{+},\ \mbox{and}\ 
	\xis\rightarrow -2\ \mbox{as}\ \s\rightarrow+\infty\,.
$$
\end{enumerate}
\end{main}

Statement (1) above means that $M$ is the limiting mass of the particle system at time $t$ as 
$\e\to 0^{+}$. Statement (3) gives a precise asymptotic equivalent of $M(t)$ as $t\to+\infty$.

\smallskip
As recalled in the previous section, if $\s=0$ in the linear Boltzmann equation ($\Xi_\e$), the
total mass of the particle system in the vanishing $\e$ limit is asymptotically equivalent to
$$
\frac{\tfrac1{2\pi}\displaystyle\iint_{\R^2\times\Sp}f^{in}(x,v)dxdv}{\pi^2 t}
$$
as $t\to+\infty$. The reason for this slow, algebraic decay is the existence of channels ---
infinite open strips included in the spatial domain $Z_\e$, i.e. avoiding all the holes. Particles
located in one such channel and moving in a direction close to the channel's direction will
not fall into a hole before exiting the channel, and this can take an arbitrarily long time as the
particles' direction approaches that of the channel. This construction based on channels 
leads to a sufficiently large fraction of the single-particle phase space and accounts for the
algebraic lower bound in (\ref{BGWBound}). The asymptotic equivalent mentioned above in 
the collisionless case $\s=0$ is a consequence of a more refined analysis based on continued
fractions given in \cite{Golse-Caglioti}.

When $\s>0$, particles whose distribution function solves the linear Boltzmann equation in
($\Xi_\e$) travel on trajectories whose direction is discontinuous in time --- more specifically,
time discontinuities are distributed under an exponential law of parameter $\s$. Obviously,
this circumstance destroys the channel structure that is responsible of the algebraic decay
of the total mass of the particle system in the collisionless case, so that one expects that the
total mass decay is faster than algebraic as $t\to+\infty$. That this decay is indeed exponential 
whenever $\s>0$ is by no means obvious: see the argument in \cite{EBMajorMass}, leading 
to an upper bound for the total mass. Statement (3) above leads to an asymptotic equivalent of 
the total mass, thereby refining the conclusions of \cite{EBMajorMass}.

\smallskip
In section 4, we give the proof of theorem $\ref{premiertheoreme}$; the evolution of the total
mass in the vanishing $\e$ limit (governing equation and asymptotic behavior as $t\to+\infty$) 
is discussed in section 5.

\section{The homogenized kinetic equation}

Our argument for the proof of Theorem \ref{premiertheoreme} is split into several steps.

\subsection{A new formulation of the transport equation}

Perhaps the most surprising feature in Theorem \ref{premiertheoreme} is the introduction of 
the extended phase space involving the additional variable $s$.

As a matter of fact, this additional variable $s$ can be used already at the level of the original
linear Boltzmann equation --- i.e. in the formulation of the problem $(\Xi_{\e})$.

Let us indeed return to the initial boundary value problem $(\Xi_{\e})$ for the linear Boltzmann 
equation.

As recalled above, the last two authors have obtained the homogenized equation corresponding 
to $(\Xi_{\e})$ in the noncollisional case $(\s=0)$ by explicitly computing the solution of the linear
Boltzmann equation for each $0<\e\ll 1$. In the collisionnal case $(\s>0)$, as recalled above, 
there is no such explicit formula giving the solution of the linear Boltzmann equation --- except the 
semi-explicit formula involving the transport process defined in \cite{PapanicoBAMS75}.

However, not all the information in that semi-explicit formula is needed for the proof of Theorem
\ref{premiertheoreme}. The additional variable $s$ is precisely the exact amount of information 
contained in that semi-explicit formula needed in the description of the homogenized process in
the limit as $\e\to 0^+$.

Consider therefore the initial boundary value problem
$$
(\Sigma_{\e})\left\{
\begin{array}{ll}
\pl_{t}F_{\e}+v\cdot\nabla_{x}F_{\e}+\pl_{s}F_{\e}+\s F_{\e}=0,\ &t,s>0,(x,v)\in Z_{\e}\times\Sp,
\\	\\
F_{\e}(t,s,x,v)=0,\ \mbox{if\ }v\cdot n_{x}>0,\ &t,s>0, (x,v) \in (\pl Z_{\e}\times\Sp),
\\	\\
F_{\e}(t,0,x,v)=\s\displaystyle\int_{0}^{\infty}KF_{\e}(t,s,x,v)ds,\ & t>0, (x,v)\in Z_{\e}\times\Sp,
\\	\\
F_{\e}(0,s,x,v)=\s e^{-\s s}f^{in}(x,v),\ &s>0, (x,v)\in Z_{\e}\times\Sp,
\end{array}
\right.
$$
with unknown $F_{\e}:=F_{\e}(t,s,x,v)$.

The relation between these two initial boundary value problems, $(\Xi_{\e})$ and $(\Sigma_{\e})$, 
is explained by the following proposition.

\begin{Prop}\label{Xif}
Assume that $f^{in}$ satisfies the assumption (\ref{condinitiale}). Then

\noindent
a) for each $\e>0$, the problem $(\Sigma_{\e})$ has a unique mild solution such
that
$$
(t,x,v)\mapsto\int_0^{+\infty}|F_\e(t,s,x,v)|ds\hbox{ belongs to }L^\infty([0,T]\times Z_\e\times\Sp)
$$
for each $T>0$;

\noindent
b) moreover
$$
0\le F_\e(t,s,x,v)\le\|f^{in}\|_{L^\infty(\R^2\times\Sp)}\s e^{-\s s}
$$
a.e. in $t,s\ge 0$, $x\in Z_\e$ and $v\in\Sp$, and
$$ 
\int_{0}^{+\infty}F_{\e}(t,s,x,v)ds=f_{\e}(t,x,v),
$$ 
for a.e. $t\geq0, x\in Z_{\e}, v\in\Sp$, where $f_\e$ is the solution of $(\Xi_{\e})$.
\end{Prop}

\begin{proof}
Applying the method of characteristics, we see that, should a mild solution $F_\e$  of the problem 
$(\Sigma_{\e})$ exist, it must satisfy
\begin{equation}
\label{FirstF}
F_{\e}(t,s,x,v)= F_{1,\e}(t,s,x,v)+F_{2,\e}(t,s,x,v),
\end{equation}
with
\begin{equation}
\begin{aligned}
F_{1,\e}(t,s,x,v) &= \ind_{s<\e \tau_\e(\frac{x}{\e},v)}\ind_{s<t}e^{-\s s}F_{\e}(t-s,0,x-vs,v)
\\
&= \ind_{s<\e \tau_\e(\frac{x}{\e},v)}\ind_{s< t}\s e^{-\s s}\int_0^{+\infty}KF_{\e}(t-s,\tau,x-sv,v)d\tau
\end{aligned}
\end{equation}
and
\begin{equation}
\label{LastF}
\begin{aligned}
F_{2,\e}(t,s,x,v) &= \ind_{t<\e \tau_\e(\frac{x}{\e},v)}\ind_{t< s}e^{-\s t}F_{\e}(0,s-t,x-vt,v)
\\
&= \ind_{t<\e \tau_\e(\frac{x}{\e},v)}\ind_{t< s}\s e^{-\s s}f^{in}(x-tv,v)
\end{aligned}
\end{equation}
a.e. in $(t,s,x,v)\in\R_+\times\R_+\times\R^2\times\Sp$. 

First, define $\mathcal{X}_T$ to be, for each $T>0$, the set of measurable functions $G$ defined on 
$\R_+\times\R_+\times Z_\e\times\Sp$ such that
$$
(t,x,v)\mapsto\int_0^{+\infty}|G(t,s,x,v)|ds\hbox{ belongs to }L^\infty([0,T]\times Z_\e\times\Sp)\,,
$$
which is a Banach space for the norm
$$
\|G\|_{\mathcal{X}_T}=\left\|\int_0^{+\infty}|G(\cdot,s,\cdot,\cdot)|ds\right\|_{L^\infty([0,T]\times Z_\e\times\Sp)}\,.
$$
Next, for each $G\in\mathcal{X}_T$, we define
$$
\mathcal{T}G(t,s,x,v):=\ind_{s<\e \tau_\e(\frac{x}{\e},v)}\ind_{s< t}\s e^{-\s s}\int_0^{+\infty}KG(t-s,\tau,x-sv,v)d\tau\,.
$$
Obviously
$$
\begin{aligned}
{}&\left\|\int_0^{+\infty}|\mathcal{T}^nG(t,s,\cdot,\cdot)|ds\right\|_{L^\infty(Z_\e\times\Sp)}
\\
&\qquad\le\s\int_0^t\left\|\int_0^{+\infty}|\mathcal{T}^{n-1}G(t_1,\tau,\cdot,\cdot)|d\tau\right\|_{L^\infty(Z_\e\times\Sp)}dt_1
\\
&\qquad\le\s^n\int_0^t\ldots\int_0^{t_{n-1}}\left\|\int_0^{+\infty}|G(t_n,s,\cdot,\cdot)|ds\right\|_{L^\infty(Z_\e\times\Sp)}dt_n\ldots dt_1\,,
\end{aligned}
$$
so that
$$
\|\mathcal{T}^nG\|_{\mathcal{X}_T}\le\frac{(\s T)^n}{n!}\|\mathcal{T}^nG\|_{\mathcal{X}_T}\,.
$$

Now $F_{1,\e}=\mathcal{T}F_\e$, so that (\ref{FirstF}) can be recast as 
$$
F_\e=F_{2,\e}+\mathcal{T}F_\e\,.
$$
This integral equation has a solution $F_\e\in\mathcal{X}_T$ for each $T>0$, given by the series
$$
F_\e=\sum_{n\ge 0}\mathcal{T}^nF_{2,\e}
$$
which is normally convergent in the Banach space $\mathcal{X}_T$ since
$$
\sum_{n\ge 0}\|\mathcal{T}^nF_{2,\e}\|_{\mathcal{X}_T}\le\sum_{n\ge 0}\frac{(\s T)^n}{n!}\|F_{2,\e}\|_{\mathcal{X}_T}<+\infty\,.
$$
Assuming that the integral equation above has another solution $F'_\e\in\mathcal{X}_T$ would
imply that
$$
F_\e-F'_\e=\mathcal{T}(F_\e-F'_\e)=\ldots=\mathcal{T}^n(F_\e-F'_\e)\,,
$$
so that
$$
\|F_\e-F'_\e\|_{\mathcal{X}_T}=\|\mathcal{T}^n(F_\e-F'_\e)\|_{\mathcal{X}_T}
	\le\frac{(\s T)^n}{n!}\|F_\e-F'_\e\|_{\mathcal{X}_T}
\to 0
$$
as $n\to+\infty$: hence $F'_\e=F_\e$. Thus we have proved statement a).

As for statement b), observe that $\mathcal{T}G\ge 0$ a.e. on $\R_+\times\R_+\times Z_\e\times\Sp$ 
if $G\ge 0$ a.e. on $\R_+\times\R_+\times Z_\e\times\Sp$. Hence, if 
$f^{in}\in L^\infty(\R^2\times\Sp)$ satisfies $f^{in}\ge 0$ a.e. on $\R^2\times\Sp$, one has 
$F_{2,\e}\ge 0$ a.e. on $\R_+\times\R_+\times Z_\e\times\Sp$, so that $\mathcal{T}^nF_{2,\e}\ge 0$
a.e. on $\R_+\times\R_+\times Z_\e\times\Sp$ and the series defining $F_\e$ is a.e. nonnegative on 
$\R_+\times\R_+\times Z_\e\times\Sp$.

Next, integrating both sides of (\ref{FirstF}) with respect to $s$, and setting 
$$
g_\e(t,x,v):=\int_0^{+\infty}F_\e(t,s,x,v)ds\,,
$$
we arrive at 
$$
\begin{aligned}
g_\e(t,x,v)
&=
\int_0^{+\infty}F_{2,\e}(t,s,x,v)ds+\int_0^{+\infty}F_{1,\e}(t,s,x,v)ds
\\
&=
\ind_{t<\e \tau_\e(\frac{x}{\e},v)}f^{in}(x-tv,v)\int_0^{+\infty}\ind_{t<s}\s e^{-\s s}ds
\\
&+
\int_0^{+\infty}\ind_{s<\e \tau_\e(\frac{x}{\e},v)}\ind_{s< t}\s e^{-\s s}
	\left(\int_0^{+\infty}KF_{\e}(t-s,\tau,x-sv,v)d\tau\right)ds
\\
&=\ind_{t<\e \tau_\e(\frac{x}{\e},v)}f^{in}(x-tv,v)e^{-\s t}
\\
&+
\int_0^te^{-\s s}\ind_{s<\e \tau_\e(\frac{x}{\e},v)}\s Kg_\e(t-s,x-sv,v)ds
\end{aligned}
$$
in which we recognize the Duhamel formula giving the unique mild solution $f_\e$ of $(\Xi_{\e})$.
Hence
$$
f_\e(t,x,v)=\int_0^{+\infty}F_\e(t,s,x,v)ds\hbox{ a.e. in }(t,x,v)\in\R_+\times Z_\e\times\Sp\,.
$$

Finally, since $(\Xi_{\e})$ satisfies the maximum principle, one has
$$
f_\e(t,x,v)\le\|f^{in}\|_{L^\infty(\R^2\times\Sp)}\hbox{ a.e. in }(t,x,v)\in\R_+\times Z_\e\times\Sp\,.
$$
Going back to (\ref{FirstF}), we recast it in the form
$$
\begin{aligned}
F_\e(t,s,x,v)&=\ind_{s<\e \tau_\e(\frac{x}{\e},v)}\ind_{s< t}\s e^{-\s s}Kf_{\e}(t-s,x-sv,v)
\\
&+
\ind_{t<\e \tau_\e(\frac{x}{\e},v)}\ind_{t< s}\s e^{-\s s}f^{in}(x-tv,v)
\\
&\le\ind_{s<\e \tau_\e(\frac{x}{\e},v)}\ind_{s< t}\s e^{-\s s}\|f^{in}\|_{L^\infty(\R^2\times\Sp)}
\\
&+
\ind_{t<\e \tau_\e(\frac{x}{\e},v)}\ind_{t< s}\s e^{-\s s}\|f^{in}\|_{L^\infty(\R^2\times\Sp)}
\\
&\le
\s e^{-\s s}\|f^{in}\|_{L^\infty(\R^2\times\Sp)}
\end{aligned}
$$
a.e. in $(t,s,x,v)\in\R_+\times\R_+\times Z_\e\times\Sp$, which concludes the proof.
\end{proof}

Observe that if
$$
F_\e(0,s,x,v)=\s e^{-\s s}f^{in}(x,v)
$$ 
is replaced with
$$
F_\e(0,s,x,v)=\Pi(s)f^{in}(x,v)
$$ 
where $\Pi$ is any probability density on $\R_{+}$ vanishing at $\infty$, the conclusion of the 
lemma above remains valid. In other words, the dependence of the solution $F_\e$ of the 
problem ($\Sigma$) upon the choice of the initial probability density $\Pi$ disappears after
integration in $s$, so that the particle distribution function $f_\e$ is indeed independent of
the choice of $\Pi$. 

The extra variable $s$ in the extended phase space has the following interpretation. Recall that 
the solution $f_{\e}$ of the linear Boltzmann equation can be expressed in terms of the transport
process (see \cite{PapanicoBAMS75}), a stochastic process involving a jump process in the 
$v$ variable, perturbed by a drift in the $x$ variable. The variable $s$ is the ``age'' of the current
velocity $v$ in that process, i.e. the time since the last jump in the $v$ variable. 

The choice $\Pi(s)=\s e^{-\s s}$ corresponds with the situation where the gas molecules have 
been evolving under the linear Boltzmann equation for $t<0$ and the holes are suddenly opened 
at $t=0$.

Before giving the proof of Theorem $\ref{premiertheoreme}$, we need to establish a few technical 
lemmas.

\subsection{The distribution of free path lengths}

A straightforward consequence of the limit in (\ref{BZLim}) is the following lemma, which accounts
eventually for the coefficient $\dot{p}(t\wedge s)/p(t\wedge s)$ in the limiting equation ($\Sigma$).

\begin{Lemma}\label{Golse-Caglioti}
Let $\tau_{\e}$ be the free path length defined in $(\ref{length})$. Then for each $t>0$
$$
\{\ind_{t<\e\tau_{\e}(\frac{x}{\e},v)}\}\rightharpoonup p(t)
$$ 
in $L^{\infty}(\R^{2}\times\Sp)$ weak-$*$ as $\e\rightarrow0^{+}$.
\end{Lemma}

\noindent
(See the definition before Theorem \ref{premiertheoreme} for the notation $\{\ind_{t<\e\tau_{\e}(\frac{x}{\e},v)}\}$.)

\begin{proof}
Since the linear span of functions $\chi\equiv\chi(x,v)$ of the form
$$
\phi(x,v)=\chi(x)\ind_I(v)\,,\quad\phi\in C^\infty_0(\R^{2})\hbox{ and $I$ an arc of $\Sp$}
$$
is dense in $L^1(\R^{2}\times\Sp)$, and the family $\ind_{\e\tau_{\e}(\frac{x}{\e},v)>t}$ is bounded
in $L^{\infty}(\R^{2}\times\Sp)$, it is enough to prove that
$$
\iint_{Z_\e\times\Sp}\phi(x,v)\ind_{\e\tau_{\e}(\frac{x}{\e},v)>t}dxdv
	\rightarrow p(t)\iint_{\R^{2}\times\Sp}\phi(x,v)dxdv\ \mbox{as\ }\e\rightarrow0\,.
$$ 
Write
\begin{eqnarray*}
\iint_{Z_\e\times\Sp}\phi(x,v)\ind_{\e\tau_{\e}(\frac{x}{\e},v)>t}dxdv
&=& \int_{Z_\e}\chi(x)\left(\int_{I}\ind_{\e\tau_{\e}(\frac{x}{\e},v)>t}dv\right)dx
\\
&=& \int_{Z_\e}\chi(x)T_{\e}\left( \frac{x}{\e}\right) dx
\end{eqnarray*}
with
$$
T_{\e}(y):=\int_{I}\ind_{\e\tau_{\e}(y,v)>t}dv\,.
$$

Obviously $T_\e$ is $1$-periodic in $y_1$ and $y_2$ and satisfies $0\le T_\e\le|I|$. Hence
$$
\ind_{d(y,\mathbbm{Z}^{2})>\e}T_{\e}(y)=\sum_{k\in\Z^{2}}\hat{T}_{\e}(k)e^{2i\pi k\cdot y}
$$
in $L^2(\R^2/\Z^2)$
with 
$$ 
\hat{T}_{\e}(k):=\int_{\max(|z_1|,|z_2|)<1/2\atop|z|>\e}T_{\e}(z)^{-2i\pi k\cdot z}dz
$$
for each $k\in\Z^{2}$. 

Then, by Parseval's identity, 
$$
\begin{aligned}
\int_{Z_\e}\chi(x)T_{\e}\left(\frac{x}{\e}\right)dx 
&= 
\int_{\R^2}\chi(x)\left(\sum_{k\in\Z^{2}}\hat{T}_{\e}(k)e^{2i\pi \frac{k\cdot x}{\e}}\right)dx
\\
&=\hat{\chi}(0)\hat{T}_{\e}(0)
	+\sum_{k\in\Z^{2}\setminus(0,0)}\hat{T}_{\e}(k)\hat{\chi}(-2\pi k/\e)\,,
\end{aligned}
$$
with 
$$ 
\hat{\chi}(\xi):=\int_{\R^{2}}\chi(x)e^{-i\xi\cdot x}dx\,.
$$

Applying again Parseval's identity,
$$
\sum_{k\in\Z^2}|\hat{T}_{\e}(k)|^2=\int_{\max(|y_1|,|y_2|)<1/2\atop|y|>\e}|T_\e(y)|^2dy\le|I|
$$
while 
$$
|\hat{\chi}(\xi)|\le\frac1{|\xi|^2}\|\nabla^2\chi\|_{L^\infty}\,,
$$
so that
$$
|\hat{\chi}(-2\pi k/\e)|\le\frac{\e^2}{4\pi^2|\xi|^2}\|\nabla^2\chi\|_{L^\infty}\,.
$$
Hence, by the Cauchy-Schwarz inequality, 
$$
\left|\sum_{k\in\Z^{2}\setminus(0,0)}\hat{T}_{\e}(k)\hat{\chi}(-2\pi k/\e)\right|^2
\le
\sum_{k\in\Z^{2}\setminus(0,0)}|\hat{T}_{\e}(k)|^2
\sum_{k\in\Z^{2}\setminus(0,0)}\frac{\e^4\|\nabla^2\chi\|^2_{L^\infty}}{16\pi^4|k|^4}
=O(\e^4)
$$
and therefore
$$
\int_{Z_\e}\chi(x)T_{\e}\left(\frac{x}{\e}\right) dx=\hat{\chi}(0)\hat{T}_{\e}(0)+O(\e^2)
$$
as $\e\to 0^+$.

By (\ref{BZLim})
$$
\hat T_\e(0)=\int_{\max(|y_1|,|y_2|)<1/2\atop|y|>\e}T_\e(y)dy\to p(t)|I|\quad\hbox{Êas }\e\to 0^+\,,
$$
so that
$$
\hat{\chi}(0)\hat{T}_{\e}(0)\to p(t)|I|\int_{\R^2}\chi(x)dx=p(t)\iint_{\R^2\times\Sp}\phi(x,v)dxdv
$$
as $\e\to 0^+$, and hence
$$
\int_{Z_\e}\chi(x)T_{\e}\left(\frac{x}{\e}\right) dx=p(t)\iint_{\R^2\times\Sp}\phi(x,v)dxdv+o(1)+O(\e^2)
$$
which entails the announced result.
\end{proof}

\subsection{Extending $f_{\e}$ by $0$ in the holes}

We begin with the equation satisfied by the (extension by $0$ inside the holes of the)
distribution function $\left\lbrace f_{\e}\right\rbrace$.

\begin{Lemma}\label{eqg}
For each $\e>0$, the function $\left\lbrace f_{\e}\right\rbrace$ satisfies
$$ 
(\pl_{t}+v\cdot\nabla_{x})\left\lbrace f_{\e}\right\rbrace
	+\s(\left\lbrace f_{\e}\right\rbrace-K\left\lbrace f_{\e}\right\rbrace)
		=(v\cdot n_x)f_{\e}\big|_{\pl Z_{\e}\times\Sp}\delta_{\pl Z_{\e}}
$$ 
in $\mathcal{D}'(\R_{+}^{*}\times\R^{2}\times\Sp),$ where $\delta_{\pl Z_{\e}}$ is the surface 
measure concentrated on the boundary of $Z_{\e}$, and $n_{x}$ is the unit normal vector at 
$x\in\pl Z_{\e}$ pointing towards the interior of $Z_{\e}$.
\end{Lemma}

\begin{proof}
One has
$$
\pl_{t}\left\lbrace f_{\e}\right\rbrace=\left\lbrace \pl_{t} f_{\e}\right\rbrace 
$$ 
and
$$
\nabla_{x}\left\lbrace f_{\e}\right\rbrace=
\left\lbrace \nabla_{x}f_{\e}\right\rbrace +f_{\e}\mid_{\pl Z_{\e}\times\Sp}\delta_{\pl Z_{\e}}n_{x}
$$ 
in $\mathcal{D}'(\R_{+}^{*}\times\R^{2}\times\Sp)$. Hence
\begin{eqnarray*}
0&=&\left\lbrace \pl_{t}f_{\e}+v\cdot\nabla_{x}f_{\e}+\s(f_{\e}-Kf_{\e})\right\rbrace  
\\
&=& \pl_{t}\left\lbrace f_{\e}\right\rbrace+v\cdot\nabla_{x}\left\lbrace f_{\e}\right\rbrace
	+(v\cdot n_{x})f_{\e}\big|_{\pl Z_{\e}\times\Sp}\delta_{\pl Z_{\e}}
		+\s(\left\lbrace f_{\e}\right\rbrace-K\left\lbrace f_{\e}\right\rbrace)
\end{eqnarray*}
in $\mathcal{D}'(\R_{+}^{*}\times\R^{2}\times\Sp)$.
\end{proof}

A straightforward consequence of the scaling considered here is that the family of Radon measures 
$$
(v\cdot n_{x})f_{\e}\big|_{\pl Z_{\e}\times\Sp}\delta_{\pl Z_{\e}}
$$ 
is controlled uniformly as $\e\rightarrow 0^{+}$, in the following manner.

\begin{Lemma}\label{mg}
For each $R>0$, the family of Radon measures 
$$
(v\cdot n_{x})f_{\e}\big|_{\pl Z_{\e}\times\Sp}\delta_{\pl Z_{\e}}\big|_{[-R,R]^{2}\times\Sp}
$$
is bounded in\footnote{For each compact subset $K$ of $\R^N$, we denote by $\mathcal{M}(K)$ the 
space of signed Radon measures on $K$, i.e. the set of all real-valued continuous linear functionals 
on $C(K)$ endowed with the topology of uniform convergence on $K$.} $\mathcal{M}([-R,R]^{2}\times\Sp)$.
\end{Lemma}

\begin{proof}
The total mass of the measure 
$$
(v\cdot n_{x})f_{\e}\big|_{\pl Z_{\e}\times\Sp}\delta_{\pl Z_{\e}}\big|_{[-R,R]^{2}\times\Sp}
$$ 
is less than or equal to
$$
2\pi\|f_{\e}\|_{L^\infty(\R_+\times Z_\e\times\Sp)}
	\|\delta_{\pl Z_{\e}}\mid_{[-R,R]^{2}}\|_{\mathcal{M}([-R,R]^{2})}
$$ 
which is itself less than or equal to
$$
2\pi\|f^{in}\|_{L^\infty(\R^2\times\Sp)}
	\left\|\delta_{\pl Z_{\e}}\mid_{[-R,R]^{2}}\right\|_{\mathcal{M}([-R,R]^{2})}\,.
$$ 
Since $\delta_{\pl Z_{\e}}\mid_{[-R,R]^{2}}$ is the union of $O\left(\left( \frac{2R}{\e}\right)^{2}\right)$ 
circles of radius $\e^{2}$, 
$$
\|\delta_{\pl Z_{\e}}\mid_{[-R,R]^{2}}\|_{\mathcal{M}([-R,R]^{2})} 
	=O\left(\left( \frac{2R}{\e}\right)^{2}\right)2\pi\e^{2}=O(1)R^{2}
$$
as $\e\to 0^+$, whence the announced result.
\end{proof}

\subsection{The velocity averaging lemmas}

As is the case of all homogenization results, the proof of Theorem \ref{premiertheoreme} is based
on the strong $L^1_{loc}$ convergence of certain quantities defined in terms of $F_\e$. In the case
of kinetic models, strong $L^1_{loc}$ compactness is usually obtained by velocity averaging --- see
for instance \cite{Agosh84,GPS,GLPS} for the first results in this direction. Below, we recall a classical
result in velocity averaging that is a special case of theorem 1.8 in \cite{Bouchut}.

\begin{Prop}\label{average}
Let $p>1$ and assume that $f_{\e}\equiv f_{\e}(t,x,v)$ is a bounded family in 
$L^{p}_{loc}(\R^{+}_{t}\times\R^{d}_{x}\times\Spd_{v})$ such that
$$
\sup_{\e}\int_{0}^{T}\iint_{B(0,R)\times\Spd}|\pl_{t}f_{\e}+v\cdot\nabla_{x}f_{\e}|dxdvdt<+\infty
$$ 
for each $T>0$ and $R>0$. Then, for each $\psi\in C(\Spd\times\Spd)$, the family $\rho_{\psi}[f_{\e}]$,
defined by
$$
\rho_{\psi}[f_{\e}](t,x,v)=\int_{\Spd}f_{\e}(t,x,v)\psi(v,w)dw
$$ 
is relatively compact in $L^{1}_{loc}(\R^{+}_{t}\times\R^{d}_{x}\times\Spd_{v})$.
\end{Prop}

A straightforward consequence of Proposition \ref{average} is the following compactness result in
$L^1_{loc}$ strong, which is the key argument in the proof of Theorem $\ref{premiertheoreme}$.

\begin{Lemma}\label{averging}
Let $f_{\e}\equiv f_{\e}(t,x,v)$ be the family of solutions of the initial boundary value problem 
$(\Xi_{\e})$. Then the families
$$
K\left\lbrace f_{\e}\right\rbrace=\left\lbrace Kf_{\e}\right\rbrace 
$$
and
$$
\int_{\Sp}\{f_\e\}dv
$$
are relatively compact in $L^{1}_{loc}(\R_{+}\times\R^{2}\times\Sp)$ strong.
\end{Lemma}

\begin{proof}
We recall that, by the Maximum Principle for $(\Xi_{\e})$, 
$$
|f_{\e}(t,x,v)|\leq\|f^{in}\|_{L^{\infty}(\R^{2}\times\Sp)}
$$ 
a.e. in $t\geq0,x\in Z_{\e}$ and $v\in\Sp$, so that
\begin{equation}
\label{Max}
\sup_{\e}\|\left\lbrace f_{\e}\right\rbrace\|_{L^{\infty}(\R_{+}\times\R^{2}\times\Sp)}
	\leq \|f^{in}\|_{L^{\infty}(\R^{2}\times\Sp)}.
\end{equation}
By Lemma $\ref{eqg}$, $\left\lbrace f_{\e}\right\rbrace$ satisfies the equation
\begin{eqnarray*}
\pl_{t}\left\lbrace f_{\e}\right\rbrace+v\cdot\nabla_{x}\left\lbrace f_{\e}\right\rbrace
	=\s(K\left\lbrace f_{\e}\right\rbrace-\left\lbrace f_{\e}\right\rbrace)
		-\delta_{\pl Z_{\e}}(v.n_{x})f_{\e}\mid_{\pl Z_{\e}\times\Sp}
\end{eqnarray*}
in $\mathcal{D}'(\R_{+}^{*}\times\R^{2}\times\Sp)$. Because of (\ref{Max}) and the fact that
the scattering kernel $k$ is a.e. nonnegative (see (\ref{kdef})), one has
$$
\begin{aligned}
\|\s(K\left\lbrace f_{\e}\right\rbrace-\left\lbrace f_{\e}\right\rbrace)\|_{L^{\infty}(\R_{+}\times\R^{2}\times\Sp)}
	&\leq\s(1+\|K1\|_{L^{\infty}(\Sp)})\|\left\lbrace f_{\e}\right\rbrace\|_{L^{\infty}(\R_{+}\times\R^{2}\times\Sp)}
\\
&=2\s\|\left\lbrace f_{\e}\right\rbrace\|_{L^{\infty}(\R_{+}\times\R^{2}\times\Sp)}
\end{aligned}
$$ 
since $K1=1$ (see again (\ref{kdef}).) Besides the family of Radon measures 
$$
\mu_{\e}=f_{\e}\mid_{\pl Z_{\e}\times\Sp}(v\cdot n_{x})\delta_{\pl Z_{\e}}
$$ 
satisfies
$$
\sup_{\e}\int_{[0,T]\times\overline{B(0,R)}\times\Sp}|\mu_{\e}|<+\infty
$$ 
for each $T>0$ and $R>0$ according to lemma $\ref{mg}$. 

Applying the Velocity Averaging result recalled above implies that the family 
$$
\int_{\Sp}g_\e dv
$$
is relatively compact in $L_{loc}^{1}(\R_{+}\times\R^{2}\times\Sp)$.

By density of $C(\Sp\times\Sp)$ in $L^2(\Sp\times\Sp)$, replacing the integral 
kernel $k$ with a continuous approximant and applying the Velocity Averaging 
Proposition \ref{average} in the same way as above, we conclude that the family 
$Kg_\e$ is also relatively compact in $L_{loc}^{1}(\R_{+}\times\R^{2}\times\Sp)$.
\end{proof}

\subsection{Uniqueness for the homogenized equation}

Consider the Cauchy problem with unknown $G\equiv G(t,s,x,v)$
$$
\left\{
\begin{array}{ll}
(\pl_{t}+v\cdot\nabla_{x}+\pl_{s})G=-\s G+\displaystyle\frac{\dot{p}(t\wedge s)}{p(t\wedge s)}G, 
	&t,s>0, x\in\R^{2},v\in\Sp,
\\	\\
G(t,0,x,v)=S(t,x,v),&t>0, (x,v)\in\R^{2}\times\Sp,
\\	\\
G(0,s,x,v)=G^{in}(s,x,v), &s>0, (x,v)\in\R^{2}\times\Sp.
\end{array}
\right.
$$
If, for a.e. $(t,s,x,v)\in\R_+\times\R_+\times\R^2\times\Sp$, the function 
$\tau\mapsto G(t+\tau,s+\tau,x+\tau v,v)$ is $C^1$ in $\tau>0$, then, since the function $p\in C^1(\R_+)$ 
and $p>0$ on $\R_+$, one has 
$$
\begin{aligned}
{}&\left(\frac{d}{d\tau}+\s-\frac{\dot{p}(t\wedge s+\tau)}{p(t\wedge s+\tau)}\right)G(t+\tau,s+\tau,x+\tau v,v)
\\
&\qquad\qquad=e^{-\s \tau}p(t\wedge s+\tau)
	\frac{d}{d\tau}\left(\frac{e^{\s \tau}G(t+\tau,s+\tau,x+\tau v,v)}{p(t\wedge s+\tau)}\right)=0\,.
\end{aligned}
$$
Hence
$$
\Gamma:\,\tau\mapsto\frac{e^{\s \tau}G(t+\tau,s+\tau,x+\tau v,v)}{p(t\wedge s+\tau)}
$$
is a constant. Therefore
$$
\Gamma(0)=\left\{
\begin{array}{ll}
\Gamma(-t)\quad&\hbox{ if }t<s,
\\
\Gamma(-s)\quad&\hbox{ if }s<t,
\end{array}
\right.
$$
so that
$$
G(t,s,x,v)=\ind_{t<s}e^{-\s t}p(t)G^{in}(s-t,x-tv,v)+\ind_{s<t}e^{-\s s}p(s)S(t-s,x-sv,v)\,.
$$

\begin{Prop}\label{eqF}
Assume that $f^{in}\in L^\infty(\R^2\times\Sp)$. Then the problem ($\Sigma$) has a unique mild solution 
$F$ such that
$$
(t,x,v)\mapsto\int_0^{+\infty}|F(t,s,x,v)|ds\hbox{ belongs to }L^\infty([0,T]\times\R^2\times\Sp)
$$
for each $T>0$. This solution satisfies
$$
\begin{aligned}
F(t,s,x,v)&=\ind_{t<s}\s e^{-\s t}p(t)f^{in}(x-tv,v)
\\
&+\ind_{s<t}\s e^{-\s s}p(s)\int_0^{+\infty}KF(t-s,\tau,x-sv,v)d\tau
\end{aligned}
$$
for a.e. $(t,s,x,v)\in\R_+\times\R_+\times\R^2\times\Sp$. 

Besides, $F\ge 0$ a.e. on $\R_+\times\R_+\times\R^2\times\Sp$ if $f^{in}\ge 0$ a.e. on $\R^2\times\Sp$.
\end{Prop}

\begin{proof}
That a mild solution of the problem ($\Sigma$), should it exist, satisfies the integral equation above 
follows from the computation presented before the proposition. 

As above, let $\mathcal{Y}_T$ be, for each $T>0$, the set of measurable functions $G$ defined a.e.
on $\R_+\times\R_+\times\R^2\times\Sp$ and such that
$$
(t,x,v)\mapsto\int_0^{+\infty}|G(t,s,x,v)|ds\hbox{ belongs to }L^\infty([0,T]\times\R^2\times\Sp)\,,
$$
which is a Banach space for the norm
$$
\|G\|_{\mathcal{Y}_T}=\left\|\int_0^{+\infty}|G(\cdot,s,\cdot,\cdot)|ds\right\|_{L^\infty([0,T]\times Z_\e\times\Sp)}\,.
$$
Next, for each $G\in\mathcal{Y}_T$, we define
$$
\mathcal{Q}G(t,s,x,v):=\ind_{s< t}\s e^{-\s s}p(s)\int_0^{+\infty}KG(t-s,\tau,x-sv,v)d\tau\,.
$$

Since $0<e^{-\s s}p(s)\le 1$, the integral kernel $k\ge 0$ on $\Sp\times\Sp$ and $K1=1$ by (\ref{kdef}),
one has
$$
\int_0^{+\infty}|\mathcal{Q}G(t,s,x,v)|ds
	\le\s\int_0^t\left\|\int_0^{+\infty}|G(t-s,\tau,\cdot,\cdot)|d\tau\right\|_{L^\infty(\R^2\times\Sp)}ds
$$
a.e. in $(t,x,v)\in[0,T]\times\R^2\times\Sp$, meaning that
$$
\begin{aligned}
{}&\left\|\int_0^{+\infty}|\mathcal{Q}^nG(t,s,\cdot,\cdot)|ds\right\|_{L^\infty(\R^2\times\Sp)}  
\\
&\qquad\le
\s\int_0^t\left\|\int_0^{+\infty}|\mathcal{Q}^{n-1}G(t_1,s,\cdot,\cdot)|ds\right\|_{L^\infty(\R^2\times\Sp)}dt_1
\\
&\qquad\le
\s^n\int_0^t\ldots\int_0^{t_{n-1}}\left\|\int_0^{+\infty}|G(t_n,s,\cdot,\cdot)|ds\right\|_{L^\infty(\R^2\times\Sp)}dt_n\ldots dt_1\,.
\end{aligned}
$$
In particular
$$
\|\mathcal{Q}^nG\|_{\mathcal{Y}_T}\le\frac{(\s T)^n}{n!}\|G\|_{\mathcal{Y}_T}\,.
$$

The integral equation in the statement of the proposition is
$$
F=F_2+\mathcal{Q}F
$$
where
$$
F_2(t,s,x,v)=\ind_{t<s}\s e^{-\s t}p(t)f^{in}(x-tv,v)\,.
$$
Therefore, arguing as in the proof of Proposition \ref{Xif}, one obtains a mild solution of ($\Sigma$) as 
the sum of the series
$$
F=\sum_{n\ge 0}\mathcal{Q}^nF_2\,,
$$
which is normally convergent in the Banach space $\mathcal{Y}_T$ for each $T>0$. 

Should there exist another mild solution, say $F'$, it would satisfy
$$
(F-F')=\mathcal{Q}(F-F')=\ldots=\mathcal{Q}^n(F-F')
$$
for all $n\ge 0$, so that
$$
\|F-F'\|_{\mathcal{Y}_T}=\|\mathcal{Q}^n(F-F')\|_{\mathcal{Y}_T}\le\frac{(\s T)^n}{n!}\|F-F'\|_{\mathcal{Y}_T}\to 0
$$
as $n\to+\infty$, which implies that $F=F'$ a.e. on $\R_+\times\R_+\times\R^2\times\Sp$.

Finally, $\mathcal{Q}F\ge 0$ a.e. on $\R_+\times\R_+\times\R^2\times\Sp$ if $F\ge 0$ a.e. on 
$\R_+\times\R_+\times\R^2\times\Sp$. Since $F$ is given by the series above, one has $F\ge 0$ 
a.e. on $\R_+\times\R_+\times\R^2\times\Sp$ whenever $f^{in}\ge 0$ a.e. on $\R^2\times\Sp$.
\end{proof}

\subsection{Proof of the homogenization theorem.}

Start from the decomposition (\ref{FirstF}) of $F_\e$. Passing to the limit as $\e\to 0^+$ in the
term $F_{2,\e}$ is easy. Indeed, by Lemma \ref{Golse-Caglioti}
\begin{equation}
\label{GClimit}
\{\ind_{t<\e\tau_{\e}(\frac{x}{\e},v)}\}\rightharpoonup p(t)
\end{equation}
in $L^{\infty}(\R^{2}_{x}\times\Sp_{v})$ weak-$*$ for each $t>0$, as $\e\rightarrow0^{+}$. 
Hence
\begin{equation}
\begin{aligned}
\{F_{2,\e}\}(t,s,x,v)=&\ind_{t<s}e^{-\s s}f^{in}(x-tv,v)\{\ind_{t<\e\tau_{\e}(\frac{x}{\e},v)}\}
\\
&\rightharpoonup\ind_{t<s}e^{-\s s}f^{in}(x-tv,v)p(t) =:F_{2}(t,s,x,v)
\end{aligned}
\end{equation}
in $L^{\infty}(\R_{t}^{+}\times\R_{s}^{+}\times\R^{2}_{x}\times\Sp_{v})$ weak-$*$ as $\e\rightarrow0^{+}$.

Next, we analyze the term $F_{1,\e}$; this is obviously more difficult as this term depends on 
the (unknown) solution $F_\e$ itself. 

We recall the uniform bound
$$
\sup_{\e}\|\left\{f_{\e}\right\}\|_{L^{\infty}(\R_{+}\times\R^{2}\times\Sp)}
	\leq \|f^{in}\|_{L^{\infty}(\R^{2}\times\Sp)}
$$
--- see Proposition \ref{Xif} b), so that, by the Banach-Alaoglu theorem
\begin{equation}
\label{conve}
\left\{f_{\e}\right\}\rightharpoonup f\ \mbox{in\ }L^{\infty}(\R_{+}\times\R^{2}\times\Sp)\ \mbox{weak-$*$}
\end{equation}
for some $f\in L^{\infty}(\R_{+}\times\R^{2}\times\Sp)$, possibly after extracting a subsequence
of $\e\to 0^+$.

Thus, applying the strong compactness Lemma $\ref{averging}$ shows that
$$
K\left\{f_{\e}\right\}\rightarrow Kf\ \mbox{in\ }L^{1}_{loc}(\R_{+}\times\R^{2}\times\Sp)\ \mbox{strong}
$$ 
as $\e\rightarrow0^{+}$.

This and the weak-$*$ convergence in Lemma \ref{Golse-Caglioti} imply that
\begin{equation} 
\begin{aligned}
\left\{F_{1,\e}\right\}=&\ind_{s<t}\s e^{-\s s}K\left\{f_{\e}\right\}(t-s,x-sv,v)\ind_{s<\e\tau_{\e}(\frac{x}{\e},v)}
\\
&\rightharpoonup\ind_{s<t}\s e^{-\s s}Kf(t-s,x-sv,v)p(s)
\end{aligned}
\end{equation}
in $L^{1}_{loc}(\R_{+}\times\R_{+}\times\R^{2}\times\Sp)$ weak as $\e\rightarrow0^{+}$. Therefore 
$$
\begin{aligned}
\left\{F_{\e}\right\}(t,s,x,v)\rightharpoonup&\ind_{s<t}\s e^{-\s s}Kf(t-s,x-sv,v)p(s)+F_{2}(t,s,x,v)
\\
&=:\tilde{F}(t,s,x,v)
\end{aligned}
$$ 
in $L^{1}_{loc}(\R_{+}\times\R_{+}\times\R^{2}\times\Sp)$ weak as $\e\rightarrow0^{+}$.

Fix $T>0$; then, for $t\in[0,T]$, one has
$$
\int_{0}^{\infty}F_{\e}(t,s,x,v)ds
	=\int_{0}^{T} F_{1,\e}(t,s,x,v)ds+e^{-\s t}f^{in}(x-tv,v)\ind_{t<\e\tau_{\e}(\frac{x}{\e},v)}
$$ 
since $F_{1,\e}$ is supported in $s\leq t\leq T$, so that
\begin{equation}
\begin{aligned}
\int_{0}^{\infty}\left\{F_{\e}\right\}(t,s,x,v)ds\rightharpoonup 
	&\int_{0}^{T}\ind_{s\leq t}Kf(t-s,x-vs,v)\s e^{-\s s}p(s)ds 
\\
&+f^{in}(x-tv,v)e^{-\s t}p(t)
\\
&=\int_{0}^{\infty}\tilde{F}(t,s,x,v)ds
\end{aligned}
\end{equation}
in $L^{1}_{loc}(\R_{+}\times\R^{2}\times\Sp)$ weakly as $\e\rightarrow0^{+}$. On the other hand
$$
\int_{0}^{\infty}\left\{F_{\e}\right\}(t,s,x,v)ds=\left\{f_{\e}\right\}(t,x,v)\rightharpoonup f(t,x,v)
$$
in $L^{\infty}(\R_{+}\times\R^{2}\times\Sp)$ weak-$*$ as $\e\rightarrow0^{+}$ --- and therefore also
in $L^{1}_{loc}(\R_{+}\times\R^{2}\times\Sp)$ weak as $\e\rightarrow0^{+}$. By uniqueness of the 
limit, we conclude that
\begin{equation}
\label{convsuite}
f(t,x,v)=\int_{0}^{\infty}\tilde{F}(t,s,x,v)ds\ \mbox{a.e. in\ } (t,x,v)\in \R_{+}\times\R^{2}\times\Sp
\end{equation}
so that $\tilde{F}$ satisfies
\begin{equation}
\begin{aligned}
\tilde{F}(t,s,x,v)&=\ind_{s<t}\s e^{-\s s}K\left(\int_{0}^{\infty}\tilde{F}(t-s,u,x-sv,\cdot)du\right)(v)p(s)
\\
&+\ind_{t<s}\s e^{-\s s}f^{in}(x-tv,v)p(t)
\end{aligned}
\end{equation}
a.e. in $(t,s,x,v)\in \R_{+}\times\R_{+}\times\R^{2}\times\Sp.$ By Proposition $\ref{eqF}$, this 
means that $\tilde{F}$ is a solution of the Cauchy problem $(\Sigma)$. 

By uniqueness of the solution of $(\Sigma)$, we conclude that $\tilde{F}=F$, and that the
whole family
$$
F_{\e}\rightharpoonup F\ \mbox{in\ }L^{1}_{loc}(\R_{+}\times\R_{+}\times\R^{2}\times\Sp)
$$ 
weakly as $\e\rightarrow0^{+}$.

Finally, $(\ref{conve})$ and $(\ref{convsuite})$ imply that
$$
\left\lbrace  f_{\e}\right\rbrace\rightharpoonup f=\int_{0}^{\infty}Fds
$$ 
in $L^{\infty}(\R_{+}\times\R^{2}\times\Sp)$ weak-$*$ as $\e \rightarrow0^{+}$, which concludes
the proof of Theorem \ref{premiertheoreme}.
\rightline{$\Box$}

\section{Asymptotic behavior of the total mass in the long time limit}

The formulation of the homogenized equation (problem ($\Sigma$)) as an integro-differential
equation set on the extended phase space involving the additional variable $s$ is of considerable
importance in understanding the asymptotic behavior of the total mass of the particle system 
as the time variable $t\to+\infty$. Indeed, this formulation implies that the total mass of the particle 
system satisfies a renewal equation, i.e. a class of integral equations for which a lot is known on 
the asymptotic behavior of the solutions in the long time limit --- see for instance in \cite{Feller} the
basic results on renewal type integral equations.

\subsection{The renewal PDE governing the mass}

We begin with a proof of Proposition \ref{renouv}.

\begin{proof}
That $\mu$ is a mild solution of the renewal PDE means that, for a.e. $(t,s)\in\R_+\times\R_+$,
$$
\begin{aligned}
\mu(t,s)=\ind_{t<s}\s e^{-\s(s-t)}e^{-\s t}p(t)+\ind_{s<t}e^{-\s s}p(s)\int_0^{+\infty}\mu(t-s,\tau)d\tau
\\
=\s e^{-\s s}p(\mi)\left(\ind_{t<s}+\ind_{s<t}\int_0^{+\infty}\mu(t-s,\tau)d\tau\right)\,.
\end{aligned}
$$

For each $T>0$, define
$$
\mathcal{R}\mu(t,s)=\ind_{s<t}\s e^{-\s s}p(s)\int_0^{+\infty}\mu(t-s,\tau)d\tau
$$
a.e. in $(t,s)\in\R_+\times\R_+$. Obviously, for each $\phi\in L^\infty([0,T];L^1(\R_+))$ and
a.e. $t\ge 0$, 
$$
\begin{aligned}
\|\mathcal{R}\phi(t,\cdot)\|_{L^1(\R_+)}
\le
\int_0^t\s e^{-\s(t-s)}p(t-s)\|\phi(s,\cdot)\|_{L^1(\R_+)}ds
\\
\le
\s\int_0^t\|\phi(s,\cdot)\|_{L^1(\R_+)}ds\,,
\end{aligned}
$$
so that, for each $n\ge 0$, one has
$$
\begin{aligned}
\|\mathcal{R}^n\phi(t,\cdot)\|_{L^1(\R_+)}
\le
\int_0^t\int_0^{t_1}\ldots\int_0^{t_{n-1}}\|\phi(t_n,\cdot)\|_{L^1(\R_+)}dt_n\ldots dt_1
\\
\le
\frac{(\s t)^n}{n!}\|\phi\|_{L^\infty([0,T];L^1(\R_+))}
\end{aligned}
$$
a.e. in $t\in\R_+$.

Arguing as in the proof of Proposition \ref{Xif}, we see that the renewal PDE has a unique
mild solution $\mu\in L^\infty([0,T];L^1(\R_+))$ for all $T>0$, which is given by the series
$$
\mu=\sum_{n\ge 0}\mathcal{R}^n(\mu^{in})
$$
where
$$
\mu^{in}(s):=\s e^{-\s s}\,.
$$

Obviously $\mathcal{R}\phi\ge 0$ a.e. on $\R_+\times\R_+$ if $\phi\ge 0$ a.e. on
$\R_+\times\R_+$, so that $\mu\ge 0$ a.e. on $\R_+\times\R_+$. Besides, for each
$T>0$, 
$$
\|\mu\|_{L^\infty([0,T];L^1(\R_+))}
	\le\sum_{n\ge 0}\frac{(\s T)^n}{n!}\|\mu^{in}\|_{L^1(\R_+)}=e^{\s T}\,,
$$
which implies in turn that
$$
0\le\mu(t,s)\le\s e^{-\s s}p(\mi)\left(\ind_{t<s}+\ind_{s<t}e^{\s T}\right)\le\s e^{\s T}e^{-\s s}
$$
a.e. in $(t,s)\in[0,T]\times\R_+$.

Finally, let $F$ be the mild solution of the problem ($\Sigma$) obtained in Proposition
\ref{Xif}. Since $F\ge 0$ a.e. on $\R_+\times\R_+\times\R^2\times\Sp$ is measurable,
one can apply the Fubini theorem to show that
$$
\begin{aligned}
m(t,s):&=\tfrac1{2\pi}\iint_{\R^2\times\Sp}F(t,s,x,v)dxdv
\\
&=
\ind_{t<s}\s e^{-\s t}p(t)\tfrac1{2\pi}\iint_{\R^2\times\Sp}f^{in}(x-tv,v)dxdv
\\
&+
\ind_{t<s}\s e^{-\s t}p(s)\int_0^\infty\tfrac1{2\pi}\iint_{\R^2\times\Sp}KF(t-s,\tau,x-sv,v)dxdvd\tau
\\
&=
\ind_{t<s}\s e^{-\s t}p(t)\tfrac1{2\pi}\iint_{\R^2\times\Sp}f^{in}(y,v)dydv
\\
&+
\ind_{t<s}\s e^{-\s t}p(s)\int_0^\infty\tfrac1{2\pi}\iint_{\R^2\times\Sp}KF(t-s,\tau,y,v)dydvd\tau
\\
&=
\ind_{t<s}\s e^{-\s t}p(t)\tfrac1{2\pi}\iint_{\R^2\times\Sp}f^{in}(y,v)dydv
\\
&+
\ind_{t<s}\s e^{-\s t}p(s)\int_0^\infty\tfrac1{2\pi}\iint_{\R^2\times\Sp}F(t-s,\tau,y,w)dydwd\tau
\\
&=
\ind_{t<s}\s e^{-\s t}p(t)\tfrac1{2\pi}\iint_{\R^2\times\Sp}f^{in}(x-tv,v)dxdv
\\
&+
\ind_{t<s}\s e^{-\s t}p(s)\int_0^\infty m(t-s,\tau)d\tau\,,
\end{aligned}
$$
where the second equality follows from the substitution $y=x-tv$ that leaves the Lebesgue
measure invariant, while the third equality follows from the identity
$$
\tfrac1{2\pi}\int_{\Sp}k(v,w)dv=1\,,
$$
which implies that
$$
\tfrac1{2\pi}\int_{\Sp}KF(t-s,\tau,y,v)dv=\tfrac1{2\pi}\int_{\Sp}F(t-s,\tau,y,w)dw\,.
$$
In other words, 
$$
m(t,s)\hbox{ satisfies the same integral equation as }\frac{\mu(t,s)}{2\pi}\iint_{\R^2\times\Sp}f^{in}(y,v)dydv.
$$

Now the solution $f_\e$ of ($\Xi_\e$) satisfies
$$
f_\e\ge 0\hbox{ a.e. on }\R_+\times\R^2\times\Sp\hbox{ and }
	\iint_{\R^2\times\Sp}f_\e(t,y,v)dydv\le\iint_{\R^2\times\Sp}f^{in}(y,v)dydv\,,
$$
which implies by Theorem \ref{premiertheoreme} that
$$
\int_{|y|\le R}\int_{\Sp}f_\e(t,y,v)dvdy
	\rightharpoonup\int_0^{+\infty}\int_{|y|\le R}\int_{\Sp}F(t,s,y,v)dvdyds\,.
$$
Hence, by Fatou's lemma
$$
\begin{aligned}
\int_0^{+\infty}\int_{|y|\le R}\int_{\Sp}F(t,s,y,v)dvdyds
	&\le\varliminf_{\e\to 0^+}\iint_{\R^2\times\Sp}f_\e(t,x,v)dxdv
\\
&\le\iint_{\R^2\times\Sp}f^{in}(y,v)dydv\,,
\end{aligned}
$$
a.e. in $t\ge 0$.

Letting $R\to+\infty$ in the inequality above, we see that $m\in L^\infty(\R_+;L^1(\R_+))$
and we have proved that the difference
$$
\Lambda(t,s)=m(t,s)-\frac{\mu(t,s)}{2\pi}\iint_{\R^2\times\Sp}f^{in}(y,v)dydv
$$
satisfies
$$
\Lambda\in L^\infty(\R_+;L^1(\R_+))\quad\hbox{ and }\quad\Lambda=\mathcal{R}\Lambda\,.
$$
By the same uniqueness argument as in the proof of Proposition \ref{eqF}, we conclude that $\Lambda=0$
a.e. on $\R_+\times\R_+$.
\end{proof}

\subsection{The total mass in the vanishing $\e$ limit}

By Theorem \ref{premiertheoreme}, the solution $f_\e$ of ($\Xi_\e$) satisfies
$$
\{f_\e\}\rightharpoonup\int_0^{+\infty}Fds\hbox{ in }L^\infty(\R_+\times\R^2\times\Sp)\hbox{ weak-$*$};
$$
therefore, checking that
$$
\iint_{\R^2\times\Sp}\{f_\e\}dxdv\rightharpoonup\int_0^{+\infty}\iint_{\R^2\times\Sp}Fdxdvds=:2\pi M(t)
$$
reduces to proving that there is no mass loss at infinity in the $x$ variable.

\begin{Lemma}\label{Mlemma}
Under the same assumptions as in Theorem \ref{premiertheoreme}
$$
\tfrac{1}{2\pi}\iint_{Z_\e\times\Sp}f_{\e}(t,x,v)dxdv
=
\tfrac{1}{2\pi}\iint_{\R^{2}\times\Sp}\{f_{\e}\}(t,x,v)dxdv\to M(t)
$$
strongly in $L^{1}_{loc}(\R_+)$ as $\e\to 0^+$.
\end{Lemma}

\begin{proof}
Going back to the proof of Proposition \ref{Xif} (whose notations are kept in the present 
discussion), we have seen that
$$
F_\e=\sum_{n\ge 0}\mathcal{T}^nF_{2,\e}\quad\hbox{Êon }\R_+\times\R_+\times Z_\e\times\Sp\,,
$$
with the notation
$$
F_{2,\e}(t,s,x,v)=\ind_{t<\e\tau_\e(\frac{x}\e,v)}\ind_{t<s}\s e^{-\s s}f^{in}(x-tv,v)\,.
$$
Since $\mathcal{T}\Phi\ge 0$ a.e. whenever $\Phi\ge 0$ a.e., the formula above implies
that
$$
F_\e\le G:=\sum_{n\ge 0}\mathcal{T}^nG_2
	\hbox{ a.e. in }(t,s,x,v)\in\R_+\times\R_+\times Z_\e\times\Sp\,,
$$
where
$$
G_2(t,s,x,v):=\ind_{t<s}\s e^{-\s s}f^{in}(x-tv,v)\,.
$$
Thus, $G$ satisfies the integral equation
$$
G=G_2+\mathcal{T}G
$$
meaning that $G$ is the mild solution of
$$
\left\{
\begin{array}{ll}
(\pl_t+v\cdot\nabla_x+\pl_s)G=-\s G\,,&\quad t,s>0\,,\,\,x\in\R^2\,,\,\,|v|=1\,,
\\	\\
G(t,0,x,v)=\s\displaystyle\int_0^{+\infty}KG(t,s,x,v)ds\,,&\quad t>0\,,\,\,x\in\R^2\,,\,\,|v|=1\,,
\\	\\
G(0,s,x,v)=f^{in}(x,v)\s e^{-\s s}\,,&\quad s>0\,,\,\,x\in\R^2\,,\,\,|v|=1\,,
\end{array}
\right.
$$
Reasoning as in Proposition \ref{Xif} shows that
$$
g(t,x,v):=\int_0^{+\infty}G(t,s,x,v)ds
$$
is the solution of the linear Boltzmann equation
$$
\left\{
\begin{array}{ll}
(\pl_t+v\cdot\nabla_x)g+\s(g-Kg)=0\,,&\quad t>0\,,\,\,x\in\R^2\,,\,\,|v|=1\,,
\\	\\
g(0,x,v)=f^{in}(x,v)\,,&\quad x\in\R^2\,,\,\,|v|=1\,.
\end{array}
\right.
$$
In view of the assumption (\ref{condinitiale}) bearing on $f^{in}$, we know that 
$$
G\ge 0\hbox{ a.e. on }\R_+\times\R_+\times\R^2\times\Sp
$$
and
$$
\begin{aligned}
\int_0^{+\infty}\iint_{\R^2\times\Sp}G(t,s,x,v)dxdvds
&=
\iint_{\R^2\times\Sp}g(t,x,v)dxdv
\\
&=
\iint_{\R^2\times\Sp}f^{in}(x,v)dxdv
\end{aligned}
$$
for each $t\ge 0$.

Summarizing, we have
$$
0\le\{F_\e\}\le G
$$
and
$$
\iiint_{\R_+\times\R^2\times\Sp}G(t,s,x,v)dsdxdv=\iint_{\R^2\times\Sp}f^{in}(x,v)dxdv<+\infty\,.
$$

Then we conclude as follows: for each $R>0$, one has
$$
\begin{aligned}
\iint_{Z_\e\times\Sp}&f_\e(t,x,v)dxdv-\int_0^{+\infty}\iint_{\R^2\times\Sp}F(t,s,x,v)dxdvds
\\
&=
\int_0^{+\infty}\int_{|x|>R}\int_{\Sp}\{F_\e\}(t,s,x,v)dvdxds
\\
&+
\int_0^{+\infty}\int_{|x|\le R}\int_{\Sp}\left(\{F_\e\}-F\right)(t,s,x,v)dvdxds
\\
&-
\int_0^{+\infty}\int_{|x|>R}\int_{\Sp}\{F\}(t,s,x,v)dvdxds=I_{R,\e}(t)+II_{R,\e}(t)+III_{R}(t)\,.
\end{aligned}
$$

First, for a.e. $t>0$, the term $I_{R,\e}(t)\to 0$ as $R\to+\infty$ uniformly in $\e>0$ since $0\le\{F_\e\}\le G$ 
and $G\in L^\infty(\R_+;L^1(\R_+\times\R^2\times\Sp))$. 

Next, the term $II_{R,\e}(t)\to 0$ strongly in $L^1_{loc}(\R_+)$ as $\e\to 0^+$ for each $R>0$ by Lemma
\ref{averging}.

Finally, since $\{F_\e\}\rightharpoonup F$ in $L^1_{loc}(\R_+\times\R_+\times\R^2\times\Sp)$ weak as 
$\e\to 0^+$, one has $0\le\{F\}\le G$, so that $F\in L^\infty(\R_+;L^1(\R_+\times\R^2\times\Sp))$. Hence 
the term $III_R(t)\to 0$ as $R\to+\infty$ for a.e. $t\ge 0$.

Thus we have proved that
$$
\iint_{Z_\e\times\Sp}f_\e(t,x,v)dxdv\to\int_0^{+\infty}\iint_{\R^2\times\Sp}F(t,s,x,v)dxdvds
$$
in $L^1_{loc}(\R_+)$ and therefore for a.e. $t\ge 0$, possibly after extraction of a subsequence of
$\e\to 0^+$.
\end{proof}

\subsection{An integral equation for $M$}

Given a function $\psi$ defined (a.e.) on the half-line $\R_+$, we abuse the notation $\psi\ind_{\R_+}$
to designate its extension by $0$ on $\R^*_-$.

Henceforth we also denote 
$$
\kappa(t):=p(t)\s e^{-\s t}\ind_{t\geq0}.
$$

\begin{Lemma}\label{Mlemmatwo}
The function $M$ defined in (\ref{Mdef}) satisfies the integral equation
$$
M(t)=\kappa*(M\ind_{\R_{+}})(t)+\tfrac{1}{2\pi\s}\kappa(t)\iint_{\R^{2}\times\Sp}f^{in}(x,v)dxdv,\ t\geq0
$$
where $*$ denotes the convolution on the real line.
\end{Lemma}

\begin{proof}
We apply the same method as for deriving the explicit representation formula for $F$ starting from 
the equation in Corollary $\ref{renouv},$ in order to find an exact formula for $m$. Indeed, by the 
method of characteristics,
\begin{eqnarray*}
m(t,s)&=&\ind_{s<t}p(s)e^{-\s s}m(t-s,0)+\ind_{t<s}p(t)e^{-\s t}m(0,s-t)
\\
&=&\ind_{s<t}p(s)\s e^{-\s s}\int_{0}^{\infty}m(t-s,u)du
\\
&+& \ind_{t<s}p(t)\s e^{-\s s}\tfrac{1}{2\pi}\iint_{\R^{2}\times\Sp}f^{in}(x,v)dxdv\,.
\end{eqnarray*}
The function $m$ satisfies therefore
\begin{equation}
\label{masse}
\begin{aligned}
m(t,s)&=\ind_{s<t}p(s)\s e^{-\s s}M(t-s)
\\
&+\ind_{t< s}p(t)\s e^{-\s s}\frac{1}{2\pi}\iint_{\R^{2}\times\Sp}f^{in}(x,v)dxdv\,.
\end{aligned}
\end{equation}
We next integrate both sides of (\ref{masse}) in $s\in \R_{+}$. By the definition (\ref{Mdef}) of $M$, we
obtain
$$
M(t)=\int_{0}^{t}\s p(s) e^{-\s s}M(t-s)ds+p(t)e^{-\s t}\tfrac{1}{2\pi}\iint_{\R^{2}\times\Sp}f^{in}(x,v)dxdv
$$
a.e. in $t\ge 0$, which is precisely the desired integral equation for $M$:
\begin{equation}
\label{Masset}
M(t)=\int_0^t\kappa(s)M(t-s)ds+\tfrac{1}{2\pi\s}\kappa(t)\iint_{\R^{2}\times\Sp}f^{in}(x,v)dxdv\,.
\end{equation}
\end{proof}
 
\subsection{An explicit representation formula for $M$}

\begin{Lemma}\label{Mlemmathree}
Let $M$ be the function defined in (\ref{Mdef}). Then
$$
M=\tfrac{1}{2\pi\s}\iint_{\R^{2}\times\Sp}f^{in}(x,v)dxdv\sum_{n \geq1}\kappa^{*n}
$$ 
with the notation 
$$
\kappa^{*n}=\underbrace{\kappa*\cdots*\kappa}_{\hbox{$n$ factors}}\,.
$$
\end{Lemma}

\begin{proof}
Observe that
\begin{equation}
\label{intk}
\begin{aligned}
\int_{0}^{+\infty}\kappa(t) dt&=\s \int_{0}^{+\infty} e^{-\s t}p(t)dt 
\\
&=1+\int_{0}^{+\infty}\dot{p}(t)e^{-\s t}dt<1\,,
\end{aligned}
\end{equation}
where the second equality results from integrating by parts the integral defining $\kappa$, and
the final inequality is implied by the fact that $p$ is a $C^1$ decreasing function.

By Lemma \ref{Mlemma}, $M\in L^1_{loc}(\R_+)$ and $M\ge 0$ a.e. on $\R_+$ since $f_\e\ge 0$
a.e. on $\R_+\times Z_\e\times\Sp$ because $f^{in}\ge 0$ a.e. on $\R^2\times\Sp$ --- see the
positivity assumption in (\ref{condinitiale}). Applying the Fubini theorem shows that
$$
\begin{aligned}
\int_0^{+\infty}\!\!M(t)dt\!=\!\int_0^{+\infty}\!\!\int_0^t\kappa(t-s)M(s)dsdt
	\!+\!\tfrac1{2\pi\s}\iint_{\R^2\times\Sp}\!\!f^{in}(x,v)dxdv\!\!\int_0^{+\infty}\kappa(t)dt
\\
=
\int_0^{+\infty}M(s)\left(\int_s^{+\infty}\!\!\kappa(t-s)dt\right)ds
	+\tfrac1{2\pi\s}\iint_{\R^2\times\Sp}f^{in}(x,v)dxdv\!\!\int_0^{+\infty}\kappa(t)dt.
\end{aligned}
$$
In other words
$$
\|M\|_{L^1(\R_+)}\le\|M\|_{L^1(\R_+)}\|\kappa\|_{L^1(\R_+)}
	+\tfrac1{2\pi\s}\iint_{\R^2\times\Sp}f^{in}(x,v)dxdv\,,
$$
so that $M\in L^1(\R_+)$ since $\|\kappa\|_{L^1(\R_+)}<1$, and
$$
\|M\|_{L^1(\R_+)}\le\frac1{2\pi\s(1-\|\kappa\|_{L^1(\R_+)})}\iint_{\R^2\times\Sp}f^{in}(x,v)dxdv\,.
$$

In particular, if
$$
\iint_{\R^2\times\Sp}f^{in}(x,v)dxdv=0
$$
then $M=0$ a.e. on $\R_+$, so that the representation formula to be established obviously
holds in this case.

Otherwise
$$
\iint_{\R^2\times\Sp}f^{in}(x,v)dxdv>0\,;
$$
define then
$$
\psi(t):=2\pi\s\left(\iint_{\R^{2}\times\Sp}f^{in}(x,v)dxdv\right)^{-1}M(t),\ t\geq0\,.
$$ 
According to Lemma \ref{Mlemmatwo}, the function $\psi$ verifies the integral equation
\begin{equation}
\label{Krein}
\psi(t)=(\kappa*(\psi\ind_{\R_{+}}))(t)+\kappa(t)\,,\quad\hbox{ a.e. in }t\ge 0\,.
\end{equation}
Applying the Fubini theorem as above shows that the linear operator
$$
\mathcal{A}:\,L^{1}(\R_{+})\ni f\mapsto\kappa*(f\ind_{\R_+})\in L^{1}(\R_{+})
$$
satisfies
$$
\|\mathcal{A}f\|_{L^1(\R_+)}\le\|\mathcal{A}\|\|f\|_{L^1(\R_+)}\quad
	\hbox{ with }\|\mathcal{A}\|=\int_{0}^{+\infty}\kappa(t) dt<1\,.
$$
Therefore $(1-\mathcal{A})$ is invertible in the class of bounded operators on $L^1(\R_+)$ with
inverse
$$
(1-\mathcal{A})^{-1}=\sum_{n\ge 0}\mathcal{A}^n\,.
$$

In particular
$$
\psi=(I-\mathcal{A})^{-1}\kappa=\sum_{n\ge 1}\kappa^{\star n}
$$
is the unique solution of the integral equation (\ref{Krein}) in $L^1(\R_+)$, which establishes
the representation formula in the lemma.
\end{proof}

\subsection{Asymptotic behavior of $M$ in the long time limit}
 
\subsubsection{The characteristic exponent $\xis$}

\begin{Lemma}
\label{xisb}
For each $\s>0$, the equation 
$$
\int_{0}^{\infty}\s e^{-(\s+\xi)t}p(t)dt=1
$$ 
with unknown $\xi$ has a unique real solution $\xis.$ This solution $\xis$ satisfies 
$$
-\s<\xis<0.
$$
\end{Lemma}

\begin{proof}
Consider the Laplace transform of the function $\kappa$ defined above:
$$
\mathcal{L}[\kappa](\xi):=\int_{0}^{\infty}\s e^{-(\s+\xi)t}p(t)dt.
$$ 
As $0<p\leq1$, $\mathcal{L}[\kappa]$ is of class $C^{1}$ on $]-\s,+\infty[$, and
$$
\dot{\mathcal{L}}[\kappa](\xi)=-\int_{0}^{\infty}\s e^{-(\s+\xi)t}tp(t)dt <0
$$ 
as $p(t)>0$ for each $t\geq0$. The function $\mathcal{L}[\kappa]$ is therefore decreasing on 
$]-\s,+\infty[$.

For each $t>0$, 
$$
\kappa(t)e^{-\xi t}\to 0^{+}\quad\hbox{ as }\xi\rightarrow+\infty\,,
$$
while
$$
\kappa(t)e^{-\xi t}\leq \s e^{-\s t}\quad\hbox{ for each }t\ge 0\,,
$$
since $0<p\le 1$. By dominated convergence, one concludes that
$$
\mathcal{L}[\kappa](\xi)\to 0^{+}\ \mbox{as\ }\xi\rightarrow+\infty.
$$ 

Besides, for each $t>0$, 
$$
\s p(t)e^{-(\s+\xi)t}\uparrow \s p(t)\,,\quad\hbox{ as }\xi\downarrow-\s^{+}\,.
$$
By monotone convergence,
$$
\mathcal{L}[\kappa](\xi)\rightarrow\s\int_0^{+\infty}p(t)dt=+\infty\,,\quad\mbox{as\ }\xi\rightarrow-\s^{+}\,.
$$ 
(Notice that the equality
$$
\int_0^{+\infty}p(t)dt=+\infty
$$
follows from the lower bound in (\ref{BGWBound}).)

By the intermediate value theorem, there exists an unique $\xis>-\s$ such that
$$
\mathcal{L}[\kappa](\xis)=1.
$$ 
Besides $\xis<0$ as $\mathcal{L}[\kappa]$ is decreasing and
$$
\mathcal{L}[\kappa](0)=\int_{0}^{\infty}\kappa(t)dt<\int_0^{+\infty}\s e^{-\s t}dt=1=\mathcal{L}[\kappa](\xis)\,,
$$
which concludes the proof.
\end{proof}

In particular 
$$
t\mapsto\kappa(t)e^{-\xis t}
$$ 
is a decreasing probability density on $\R_+$. 

\subsubsection{The Renewal Equation}

It remains to prove statement (3) in Theorem \ref{secondtheorem}.

First, for each $\lambda\in\R$ and each locally bounded measurable function $f:\R\mapsto\R$ supported
in $\R_+$, denote 
$$
f_{\la}(t):=e^{\la t}f(t)\ \mbox{for each\ }t\in \R\,.
$$ 
Notice that for each such $f,g$, we have
$$
e^{\la t}(f*g)(t)= (f_{\la}*g_{\la})(t)\ \mbox{for each\ }t\in\R.
$$ 
Hence, if $\psi$ is a solution of the integral equation $(\ref{Krein})$, the function $\psi_{-\xis}$ 
satisfies 
\begin{equation}
\label{RenewalEq}
\psi_{-\xis}(t)= (\kappa_{-\xis}*\psi_{-\xis})(t)+\kappa_{-\xis}\,,
\end{equation}
which is a renewal integral equation, in the sense of \cite{Feller}.

Moreover, as noticed above, $\kappa_{-\xis}$ is a decreasing probability density on $\R_{+}$, so
that in particular $\kappa_{-\xis}$ is directly Riemann integrable (see \cite{Feller} pp. 348-349). 
Thus, applying Theorem 2 on p. 349 in \cite{Feller} shows that
\begin{equation}
\label{asymppsi}
\psi(t)e^{-\xis t}\rightarrow \frac{1}{\displaystyle\int_{0}^{\infty}t\kappa(t)e^{-\xis t}dt}
	\quad\mbox{as\ }t\rightarrow+\infty.
\end{equation}
By definition of $\psi$, this is precisely the asymptotic behavior of $M$ in Theorem 2 (3). 

\subsection{Two important limiting cases for $\xis$}

We conclude our proof of Theorem \ref{secondtheorem} with a discussion of the asymptotic behavior 
of $\xis$ (statement (4) of Theorem \ref{secondtheorem}) in the two following regimes: 
\begin{enumerate}
\item the collisionless regime $\s\rightarrow0^{+},$ and

\item the highly collisional regime $\s\rightarrow+\infty.$
\end{enumerate}

\begin{proof}[End of the proof of Theorem \ref{secondtheorem}]
Denote for the sake of simplicity $\ls:=\s+\xis$. Establishing that $\xis\sim -\s$ as $\s\to 0^{+}$ 
amounts to proving that $\ls=o(\s)$. First, notice that, since $-\s<\xis$,
$$
0<\ls<\s
$$ 
so $\ls\to 0^{+}$ as $\s\to 0^{+}$. Keeping this in mind, we have
\begin{equation}
\label{lisa}
\int_{0}^{+\infty}e^{-\ls t}p(t)dt=\frac{1}{\s}
\end{equation}
by definition of $\xis$. Substituting $z=\ls t$ in the integral above, we obtain:
$$
0<\frac{\ls}{\s}=\int_{0}^{+\infty}e^{-z}p(z/\ls)dz.
$$ 
Since $\ls\to 0^{+}$ as $\s\to 0^{+}$ and $p(t)\to 0^{+}$ as $t\to +\infty$, one has $p(z/\ls)\to 0^{+}$ 
as $\s\to 0^{+}$. Besides $0\le e^{-z}p(z/\ls)\le e^{-z}$ so that, by dominated convergence 
$$
\frac{\ls}{\s}\to 0\ \mbox{as\ }\s\to 0^{+}.
$$
This establishes the asymptotic behavior of $\xis$ in the collisionless regime.

As for the highly collisional regime, we return to the equation (\ref{lisa}) defining $\xis$ (written 
in terms of $\ls$):
$$
\begin{aligned}
1&=\s\int_{0}^{+\infty}e^{-\ls t}p(t)dt
\\
&=\ls\int_{0}^{\infty}e^{-\ls t}p(t)dt-\xis\int_{0}^{\infty}e^{-\ls t}p(t)dt
\\
&=1+ \int_{0}^{\infty}e^{-\ls t}\dot{p}(t)dt -\xis\int_{0}^{\infty}e^{-\ls t}p(t)dt
\end{aligned}
$$
where the last equality follows from integrating by parts the first integral on the left hand side. 
Therefore
$$
\xis=\frac{\displaystyle\int_{0}^{\infty}e^{-\ls t}\dot{p}(t)dt}{\displaystyle\int_{0}^{\infty}e^{-\ls t}p(t)dt},
$$ 
or, after substituting $t'=\ls t$,
\begin{equation}
\label{xis}
\xis=\frac{\displaystyle\int_{0}^{\infty}e^{-t}\dot{p}(t/\ls)dt}{\displaystyle\int_{0}^{\infty}e^{-t}p(t/\ls)dt}.
\end{equation}
Equation $(\ref{lisa})$ shows that $\ls\to+\infty$ as $\s\to+\infty$. Passing to the limit in the right-hand 
side of (\ref{xis}), we find, by dominated convergence
$$
\xis\to\frac{\displaystyle\int_{0}^{\infty}e^{-t}\dot{p}(0)dt}{\displaystyle\int_{0}^{\infty}e^{-t}p(0)dt}
	=\dot{p}(0)\quad\hbox{ as }\s\to+\infty\,.
$$ 
Indeed $p$ is decreasing and convex, as can be verified for instance on the Boca-Zaharescu
explicit formula\footnote{In space dimension higher than $2$, one can show that the analogue 
of $p$ is also nonincreasing and convex, by using a variant of a formula due to L.A. Santal\`o 
established in \cite{DDG}, for want of a en explicit formula giving the limiting distribution of free
path lengths.} (\ref{BZone})-(\ref{BZtwo}) for $p$, so that
$$
0\le -\dot{p}(t)\le-\dot{p}(0)\,,\quad\hbox{ for each }t\ge 0\,.
$$
We conclude by observing that the same explicit formulas of Boca-Zaharescu \cite{BZ} imply 
that
$$
\dot{p(0)}=-2\,.
$$
\end{proof}

\section{Final remarks and open problems}

The present work provides a complete description of the homogenization of the linear Boltzmann
equation for monokinetic particles in the periodic system of holes of radius $\e^2$ centered at the
vertices of the square lattice $\e\mathbb{Z}^2$ (Theorem \ref{premiertheoreme}.) In particular, we
have given an asymptotic equivalent of exponential type of the total mass of the particle system
in the long time limit (Theorem \ref{secondtheorem}.)

Since the discussion in the present paper is restricted to the two dimensional setting, it would be
useful to extend the results above to the case of higher space dimensions, and to lattices other
than the square or cubic lattice. Most of the arguments considered here can be adapted to these
more general cases; however, the analogue of the distribution of free path lengths (the function
$p(t)$) is not known explicitly so far. See \cite{EBDimQcq} for these more general cases.

Otherwise, it would also be interesting to investigate other scalings than the Boltzmann-Grad
type scaling considered here --- holes of radius $\e^2$ centered at the vertices of a square
lattice whose fundamental domain is a square of sise $\e$ in the case of space dimension $2$.
Typically, one would like to mix the homogenization procedure considered in the present work
with the assumption of a highly collisional regime $\s\gg 1$, so that the size of the holes and 
the distance between neighboring holes are scaled in a way that differs from the one considered
here. We hope to return to this problem in a forthcoming publication.

Finally, the homogenization result considered in the present paper raises an interesting question, 
of quite general bearing. Usually, homogenization is a limiting process leading to a macroscopic
description of some material that is known at the microscopic scale. In the problem considered
here, it has been necessary to use a more detailed description of the particle system than that
provided by the linear Boltzmann equation (problem ($\Xi_\e$) set in the extended phase space 
that involves the additional variable $s$.)

In other words, the formulation of the macroscopic homogenization limit for the linear Boltzmann
equation considered here involves remnants of an \textit{even more microscopic description} of 
the system than the linear Boltzmann equation itself --- namely the extended phase space and 
the additional variable $s$.
 
We do not know whether this phenomenon (i.e. the need for a more microscopic description of a
system to arrive at the formulation of a homogenized equation for that system) can be observed 
in homogenization problems other than the one considered here --- for instance in the case of 
equations other than those found in context of kinetic theory.



\begin{thebibliography}{99}

\bibitem{Agosh84} 
Agoshkov, V.I., 
\textit{Spaces of functions with differential-difference characteristics and the smoothness 
of solutions of the transport equation},
Dokl. Akad. Nauk SSSR \textbf{276} (1984), no. 6, 1289--1293.

\bibitem{EBMajorMass}
Bernard, E.,
\textit{On the mass loss rate for the transport process in a domain with a periodic system 
of holes},
preprint.

\bibitem{EBDimQcq}
Bernard, E.,
work in preparation.

\bibitem{BZ}
Boca, F., Zaharescu, A.,
\textit{The Distribution of the Free Path Lengths in the Periodic Two-Dimensional Lorentz Gas 
in the Small-Scatter Limit}.
Commun. Math. Phys. \textbf{269} (2007), 425--471.

\bibitem{Bouchut}
Bouchut F., Golse F., Pulvirenti M.,
``Kinetic Equations and Asymptotic Theory", L. Desvillettes and B. Perthame eds.
Series in Applied Mathematics no. 4. 
Gauthier-Villars, Editions Scientifiques et M\'edicales Elsevier, Paris 2000.

\bibitem{BGW}
Bourgain, J., Golse, F., Wennberg, B.,
\textit{On the distribution of free path lengths for the periodic Lorentz gas}. 
Commun. Math. Phys. {\bf 190} (1998), 491--508.

\bibitem{Golse-Caglioti}
Caglioti, E., Golse, F.,
\textit{On the Distribution of Free Path Lengths for the Periodic Lorentz Gas III}.
Commun. Math. Phys. \textbf{236} (2003), 199--221.

\bibitem{CagliotiGolse08}
Caglioti, E., Golse, F.,
\textit{The Boltzmann-Grad limit of the periodic Lorentz gas in two space dimensions}. 
C. R. Math. Acad. Sci. Paris \textbf{346} (2008), 477--482.

\bibitem{Chandra60} 
Chandrasekhar, S.,
``Radiative Transfer".
Oxford Univ. Press, London (1950).

\bibitem{Dahl}
Dahlqvist, P.,
\textit{The Lyapunov exponent in the Sinai billiard in the small scatterer limit}.
Nonlinearity \textbf{10} (1997), 159--173.

\bibitem{DDG}
Dumas, H.S., Dumas, L., Golse, F.,
\textit{Remarks on the notion of mean free path for a periodic array of spherical obstacles}.
J. Stat. Phys. {\bf 87} (1997), 943--950.

\bibitem{Feller}
Feller, W.,
``An Introduction to Probability Theory and Its Applications'' Vol. II.
Wiley Series in Probability and Mathematical Statistics (1966).

\bibitem{Gallavotti1969}
Gallavotti, G.,
\textit{Divergences and approach to equilibrium in the Lorentz and the wind--tree--models}.
Phys. Rev. (2) {\bf 185} (1969), 308--322.

\bibitem{Gallavotti1972}
Gallavotti, G.,
\textit{Rigorous theory of the Boltzmann equation in the Lorentz gas}.
Nota interna no. 358, Istituto di Fisica, Univ. di Roma (1972).
Available as preprint mp-arc-93-304.

\bibitem{Gallavotti1999}
Gallavotti, G.,
``Statistical mechanics: a short treatise", Springer, Berlin-Heidelberg (1999).

\bibitem{Golse2008}
Golse, F., 
\textit{On the periodic Lorentz gas in the Boltzmann-Grad scaling}. 
Ann. Facult\'e des Sci. Toulouse \textbf{17} (2008), 735--749.

\bibitem{GW}
Golse F., Wennberg B.,
\textit{On the distribution of free path lengths for the periodic Lorentz gas II}.
M2AN Mod\'el. Math. et Anal. Num\'er. {\bf 34} (2000), 1151--1163.

\bibitem{GLPS}
Golse F., Lions P.-L., Perthame B., Sentis R.,
\textit{Regularity of the moments of the solution of a transport equation}.
J. Funct. Anal. {\bf 76} (1988), 110--125.

\bibitem{GPS}
Golse F., Perthame B., Sentis R., 
\textit{Un r\'esultat de compacit\'e pour les \'equations du transport et application 
au calcul de la limite de la valeur propre principale d'un op\'erateur de transport}.
C.R. Acad. Sci. S\'erie I, {\bf 301} (1985), 341--344.

\bibitem{Iannelli}
Ianelli, M.,
``Mathematical Theory of Age-Structured Population Dynamics".
Applied Math. Monographs, CNR, Giardini Editori e Stampatori, Pisa (1995).


\bibitem{MarklofStromb2007}
Marklof, J., Str\"ombergsson, A.,
\textit{The distribution of free path lengths in the periodic Lorentz gas and 
related lattice point problems}.
Preprint arXiv:0706.4395, to appear in Ann. Math..

\bibitem{MarklofStromb2008}
Marklof, J., Str\"ombergsson, A.,
\textit{The Boltzmann-Grad limit of the periodic Lorentz gas}.
Preprint arXiv:0801.0612.


\bibitem{PapanicoBAMS75}
Papanicolaou, G.C.,
\textit{Asymptotic analysis of transport processes}.
Bull. Amer. Math. Soc. \textbf{81} (1975), 330--392.

\bibitem{Thieme}
Thieme, H.R.,
``Mathematics in population biology".
Woodstock Princeton University Press, Princeton NJ (2003).

\end{thebibliography}
\end{document}